\documentclass{amsart}

\usepackage[all]{xy}
\usepackage{amssymb}
\usepackage{amsthm}
\usepackage[headings]{fullpage}
\usepackage{color}
\usepackage[normalem]{ulem}
\usepackage{stmaryrd}			%for double brackets
\usepackage[unicode]{hyperref}	%for hyperlinks
\usepackage{mathrsfs}			%for mathscr

\newcommand{\QQ}{\mathbb Q}
\newcommand{\ZZ}{\mathbb Z}

\newcommand{\BB}{\mathbb B}
\newcommand{\LL}{\mathcal{L}}
\newcommand{\CC}{\mathbb C}
\newcommand{\Qp}{\QQ_p}
\newcommand{\Zp}{\ZZ_p}
\newcommand{\calH}{\mathcal{H}}

\newcommand{\calO}{\mathcal{O}}
\newcommand{\DD}{\mathbb D}
\newcommand{\vp}{\varphi}
\newcommand{\Dcris}{\DD_\mathrm{cris}}

\newcommand{\Hom}{\mathrm{Hom_{cts}}}
\newcommand{\Qpn}{\mathbb{Q}_{p,n}}
\DeclareMathOperator{\Char}{Char}

\DeclareMathOperator{\Tw}{Tw}

\DeclareMathOperator{\Sym}{Sym}
\DeclareMathOperator{\Iw}{Iw}
\DeclareMathOperator{\Gal}{Gal}
\DeclareMathOperator{\ord}{ord}

\DeclareMathOperator{\cris}{cris}
\DeclareMathOperator{\Fil}{Fil}

\newcommand{\OO}{\mathcal{O}}

\newcommand{\symV}{\Sym^m(V_f)}
\newcommand{\HH}{\mathcal{H}}
\newcommand{\RR}{\mathbb R}

\DeclareMathOperator{\Ind}{Ind}
\DeclareMathOperator{\cor}{cor}

\newcommand{\IndQK}{\Ind_K^\QQ}
\newtheorem{theorem}{Theorem}[section]
\newtheorem{proposition}[theorem]{Proposition}
\newtheorem{lemma}[theorem]{Lemma}
\newtheorem{corollary}[theorem]{Corollary}
\newtheorem{remark}[theorem]{Remark}
\newtheorem{definition}[theorem]{Definition}
\newtheorem{hypothesis}[theorem]{Hypothesis}
\newtheorem{conjecture}[theorem]{Conjecture}

\newcommand{\ol}[1]{\overline{#1}}
\newcommand{\triv}{\mathbf{1}}	%The trivial character
\newcommand{\KL}{\text{KL}}
\newcommand{\mf}[1]{\mathfrak{#1}}
\newcommand{\mc}[1]{\mathcal{#1}}
\newcommand{\wt}[1]{\widetilde{#1}}
\newcommand{\dR}{\text{dR}}

\renewcommand{\mod}[1]{\text{ (mod }#1)}
\renewcommand{\varepsilon}{\epsilon}
\renewcommand{\AA}{\mathbb{A}}
\DeclareMathOperator{\Frac}{Frac}
\DeclareMathOperator{\Frob}{Frob}
\DeclareMathOperator{\gr}{gr}
\DeclareMathOperator{\sgn}{sgn}
\DeclareMathOperator{\ad}{ad}
\renewcommand{\Col}{\operatorname{Col}}
\DeclareMathOperator{\im}{im}

\newcommand{\fl}{\mathfrak{l}}

\hypersetup{pdfauthor={Robert Harron and Antonio Lei},pdftitle={Iwasawa theory for symmetric powers of CM modular forms at non-ordinary primes},pdfkeywords={\textit{p}-adic \textit{L}-functions, Iwasawa theory, Main conjecture, \textit{L}-invariants, CM modular forms, supersingular, symmetric powers}}

\begin{document}

\title{I\MakeLowercase{wasawa theory for symmetric powers of }CM\MakeLowercase{ modular forms at non-ordinary primes}}

\author{Robert Harron}
\address[Harron]{
Department of Mathematics\\
Van Vleck Hall\\
University of Wisconsin--Madison\\
Madison, WI 53706\\
USA
}
\email{rharron@math.wisc.edu}

\author{Antonio Lei}
\address[Lei]{Department of Mathematics and Statistics\\
Burnside Hall\\
McGill University\\
Montreal, QC\\
Canada H3A 0B9}
\email{antonio.lei@mcgill.ca}
\thanks{The second author is supported by a CRM-ISM postdoctoral fellowship.}

\begin{abstract}
Let $f$ be a cuspidal newform with complex multiplication (CM) and let $p$ be an odd prime at which $f$ is non-ordinary. We construct admissible $p$-adic $L$-functions for the symmetric powers of $f$, thus verifying general conjectures of Dabrowski and Panchishkin in this special case. We also construct their ``mixed'' plus and minus counterparts and prove an analogue of Pollack's decomposition of the admissible $p$-adic $L$-functions into mixed plus and minus $p$-adic $L$-functions. On the arithmetic side, we define corresponding mixed plus and minus Selmer groups. We unite the arithmetic with the analytic by first formulating the Main Conjecture of Iwasawa Theory relating the plus and minus Selmer groups with the plus and minus $p$-adic $L$-functions, and then proving the exceptional zero conjecture for the admissible $p$-adic $L$-functions. The latter result takes advantage of recent work of Benois, while the former uses recent work of the second author, as well as the main conjecture of Mazur--Wiles.
\end{abstract}


\maketitle

%++++++++++++++++++++++++++++++++++++++++++++++++++++++++++++++
%++++++++++++++++++++++++++++++++++++++++++++++++++++++++++++++
 
 \tableofcontents
 
\section{Introduction}
	This paper can be considered as a simultaneous sequel to the second author's article \cite{lei10} on the main conjecture for the symmetric square of a non-ordinary CM modular form and the first author's article \cite{harron12} on the exceptional zero conjecture for the symmetric powers of ordinary CM modular forms. Let $f$ be a normalised newform of weight $k\geq2$, level $\Gamma_1(N)$, and Nebentypus $\epsilon$, and let $\rho_f$ be the $p$-adic Galois representation attached to $f$. The study of the Iwasawa theory of the symmetric powers of $f$ began in the 1980s with the symmetric square case in the work of Coates--Schmidt and Hida (e.g.~\cite{coatesschmidt,schmidt88,Hi88,Hi90}). More precisely, they study the adjoint representation $\ad^0\!\rho_f\cong\Sym^2\!\rho_f\otimes\det^{-1}\!\rho_f$ in the ordinary case, constructing one- and two-variable analytic $p$-adic $L$-functions and stating a main conjecture (for $f$ corresponding to an elliptic curve). More recently, still in the ordinary case, Urban (\cite{Urban??}) has announced a proof of one divisibility in the main conjecture for $\ad^0\!\rho_f$ under some technical hypotheses. For general higher symmetric powers, even Deligne's conjecture on the algebraicity of special values of $L$-functions remains open, and consequently the analytic $p$-adic $L$-functions interpolating these values have yet to be constructed. The case where $f$ has complex multiplication is however accessible. In the ordinary setting, the analytic $p$-adic $L$-functions were constructed by the first author in \cite{harron12} and the exceptional zero conjecture was proved there as well. At non-ordinary primes, the second author constructed $p$-adic $L$-functions and formulated the Main Conjecture of Iwasawa Theory for the symmetric square in \cite{lei10}. This article deals with the non-ordinary CM case for all symmetric powers. Specifically, we construct the analytic $p$-adic $L$-functions for $\rho_m:=\Sym^m\rho_f\otimes\det^{-\lfloor m/2\rfloor}\rho_f$ and prove the corresponding exceptional zero conjecture. We also formulate a Main Conjecture of Iwasawa Theory in this setting.
	
	To begin, Section~\ref{sec:NandC} describes our notation and normalizations, while Section~\ref{sec:sympower} develops some basic facts about $\rho_m$. In Section~\ref{sec:plusminus}, we construct what we refer to as ``mixed'' plus and minus $p$-adic $L$-functions for $\rho_m$. The construction of these proceeds as in \cite{harron12} taking advantage of the decomposition of $\rho_m$ into a direct sum of modular forms and Dirichlet characters. Each modular form that shows up has plus and minus $p$-adic $L$-functions constructed by Pollack in \cite{pollack03} and we define $2^{\lceil m/2\rceil}$ mixed plus and minus $p$-adic $L$-functions for $\rho_m$ by independently choosing a sign for each modular form. We then use these mixed $p$-adic $L$-functions in Section~\ref{sec:MC} to formulate the Main Conjecture of Iwasawa Theory for $\rho_m$. This also involves defining mixed plus and minus Selmer groups which uses the decomposition of $\rho_m$ and the work of the second author in \cite{lei09}. In Section \ref{sec:exceptional}, we construct $2^{\lceil m/2\rceil}$ ``admissible'' $p$-adic $L$-functions for $\rho_m$, verifying a conjecture of Dabrowski \cite[Conjecture 1]{dabrowski11} and Panchishkin \cite[Conjecture 6.2]{panchishkin94}. It should be noted that for $m\geq2$, none of these $p$-adic $L$-functions is determined by its interpolation property. We go on to prove decompositions of the admissible $p$-adic $L$-functions into a linear combination of the mixed $p$-adic $L$-functions and mixed $p$-adic logarithms analogous to those developed by Pollack in \cite{pollack03}. Finally, in Section~\ref{subsec:exceptional}, we locate the trivial zeroes of the admissible $p$-adic $L$-functions of $\rho_m$ and prove the related exceptional zero conjecture showing that the analytic $L$-invariants are given by Benois's arithmetic $L$-invariant defined in \cite{benois11}. This relies on recent work of Benois \cite{benois??} computing the analytic and arithmetic $L$-invariants of modular forms at near-central points in the crystalline case. The trivial zero phenomenon presents an interesting feature in our situtation: the order of the trivial zero grows with the symmetric power, thus providing examples of trivial zeroes of order greater than one for motives for $\QQ$. This also has ramifications for a question raised by Greenberg in \cite[p.~170]{G94}: Is the $L$-invariant independent of the symmetric power? Our result shows that the answer must be \textit{no}, unless every CM modular form has $L$-invariant equal to 1. The latter option seems unlikely.

%+++++++++++++++++++++++++++++++++++++++++++++++++++++++++++++

\section{Notation and conventions}\label{sec:NandC}

\subsection{Extensions by \texorpdfstring{$p$}{p}-power roots of unity} 
   Throughout this paper, $p$ is a fixed odd prime. If $K$ is a finite extension of $\QQ$ or $\QQ_p$, then we write $G_K$ for its absolute Galois group, $\mathcal{O}_K$ for its ring of integers, and $\chi$ for the $p$-adic cyclotomic character of $G_K$. We fix an embedding $\iota_\infty:\ol{\QQ}\rightarrow\CC$ and write $\Frob_\infty$ for the induced complex conjugation in $G_\QQ$. We also fix embeddings $\iota_\ell:\ol{\QQ}\rightarrow\ol{\QQ}_\ell$. These embeddings fix embeddings of $G_{\QQ_\ell}$ as decomposition groups in $G_\QQ$. Let $\Frob_\ell$ denote an arithmetic Frobenius element in $G_{\QQ_\ell}$. If $K$ is a number field, $\AA_K$ (resp., $\AA_{K,f}$ and $\AA_{K,\infty}$) denotes its adele ring (resp., its ring of finite adeles and its ring of infinite adeles).

For an integer $n\ge0$, we write $K_n$ for the extension $K(\mu_{p^n})$ where $\mu_{p^n}$ is the set of $p^n$th roots of unity, and $K_\infty$ denotes $\bigcup_{n\ge1} K_n$. When $K=\QQ$, we write $k_n=\QQ(\mu_{p^n})$ instead, and when $K=\Qp$, we write $\Qpn=\Qp(\mu_{p^n})$. Let $G_n$ denote the Galois group $\Gal(\Qpn/\Qp)$ for $0\le n \le\infty$. Then, $G_\infty\cong\Delta\times\Gamma$ where $\Delta=G_1$ is a finite group of order $p-1$ and $\Gamma=\Gal(\QQ_{p,\infty}/\QQ_{p,1})\cong\Zp$. The cyclotomic character $\chi$ induces an isomorphism $G_\infty\cong\ZZ_p^\times$ identifying $\Delta$ with $\mu_{p-1}$ and $\Gamma$ with $1+p\ZZ_p$. We fix a topological generator $\gamma_0$ of $\Gamma$.

 We identify a Dirichlet character $\eta:\ZZ/N\ZZ\rightarrow\CC^\times$ with a Galois character $\eta:G_\QQ\rightarrow\ol{\QQ}_p^\times$ via $\eta(\Frob_\ell)=\iota_p\iota_\infty^{-1}\eta(\ell)$ noting that $\eta(\Frob_\infty)=\eta(-1)$.
  
%++++++++++++++++++++++++++++++++++++++++++++

\subsection{Iwasawa algebras and power series}\label{sec:iwasawaalgebras}
 Given a finite extension $E$ of $\Qp$, $\Lambda_{\calO_{E}}(G_\infty)$ (resp., $\Lambda_{\calO_{E}}(\Gamma)$) denotes the Iwasawa algebra of $G_\infty$ (resp., $\Gamma$) over $\calO_{E}$. We write $\Lambda_{E}(G_\infty)=\Lambda_{\calO_{E}}(G_\infty)\otimes_{\calO_{E}}E$ and $\Lambda_{E}(\Gamma)=\Lambda_{\calO_{E}}(\Gamma)\otimes_{\calO_{E}}E$. There is an involution $f\mapsto f^\iota$ on each of these algebras given by sending each group element to its inverse. If $M$ is a finitely generated $\Lambda_{E}(\Gamma)$-torsion module, we write $\Char_{\Lambda_{E}(\Gamma)}(M)$ for its characteristic ideal. The Pontryagin dual of $M$ is written as $M^\vee$.

Given a module $M$ over $\Lambda_{\calO_{E}}(G_\infty)$ or $\Lambda_{E}(G_\infty)$ and a character $\nu:\Delta\rightarrow\Zp^\times$, $M^\nu$ denotes the $\nu$-isotypical component of $M$. Explicitly, there is an idempotent $\pi_\nu\in\Lambda_{\calO_E}(G_\infty)$ given by
\[ \pi_\nu=\frac{1}{p-1}\sum_{\sigma\in\Delta}\nu(\sigma)\sigma^{-1}
\]
that projects $M$ onto $M^\nu$. If $m\in M$, we sometimes write $m^\nu$ for $\pi_\nu m$. Note that the character group of $\Delta$ is generated by the Teichm\"uller character $\omega:=\chi|_\Delta$.

Let $r\in\RR_{\ge0}$. We define
\[
\HH_r=\left\{\sum_{n\geq0,\sigma\in\Delta}c_{n,\sigma}\cdot\sigma\cdot X^n\in\CC_p[\Delta]\llbracket X\rrbracket:\sup_{n}\frac{|c_{n,\sigma}|_p}{n^r}<\infty\ \forall\sigma\in\Delta\right\}
\]
where $|\cdot|_p$ is the $p$-adic norm on $\CC_p$ such that $|p|_p=p^{-1}$. We write $\HH_\infty=\cup_{r\ge0}\HH_r$ and $\HH_r(G_\infty)=\{f(\gamma_0-1):f\in\HH_r\}$
 for $r\in\RR_{\ge0}\cup\{\infty\}$. In other words, the elements of $\HH_r$ (respectively $\HH_r(G_\infty)$) are the power series in $X$ (respectively $\gamma_0-1$) over $\CC_p[\Delta]$ with growth rate $O(\log_p^r)$. If $F,G\in\HH_\infty$ or $\HH_\infty(G_\infty)$ are such that $F=O(G)$ and $G=O(F)$, we write $F\sim G$.

Given a subfield $E$ of $\CC_p$, we write $\HH_{r,E}=\HH_r\cap E[\Delta]\llbracket X\rrbracket$ and similarly for $\HH_{r,E}(G_\infty)$. In particular, $\HH_{0,E}(G_\infty)=\Lambda_{E}(G_\infty)$. 

Let $n\in\ZZ$. We define the $E$-linear map $\Tw_n$ from $\HH_{r,E}(G_\infty)$ to itself to be the map that sends $\sigma$ to $\chi(\sigma)^n\sigma$ for all $\sigma\in G_\infty$. It is clearly bijective (with inverse $\Tw_{-n}$).

If $\displaystyle h=\sum_{\sigma\in\Delta}c_{n,\sigma}\cdot\sigma\cdot(\gamma_0-1)^n\in\HH_\infty(G_\infty)$ and $\lambda\in\Hom(G_\infty,\CC_p^\times)$, we write
\[
\lambda(h)=\sum_{\sigma\in\Delta} c_{n,\sigma}\cdot\lambda(\sigma)\cdot(\lambda(\gamma_0)-1)^n\in\CC_p.
\]

%+++++++++++++++++++++++++++++++++++++++++++

\subsection{Crystalline representations and Selmer groups}
We write $\BB_{\cris}\subseteq\BB_{\rm dR}$ for the rings of Fontaine (for $\QQ_p$) and $\vp$ for the Frobenius acting on $\BB_{\cris}$. Recall that there exists an element $t\in\BB_{\cris}$ such that $\vp(t)=pt$ and $g\cdot t=\chi(g)t$ for $g\in G_{\Qp}$.

Let $V$ be an $E$-linear $p$-adic representation of $G_{\Qp}$ for some finite extension $E/\Qp$. We denote the crystalline Dieudonn\'{e} module and its dual by $\DD_{\cris}(V)=\operatorname{Hom}_{\Qp[G_{\Qp}]}(V,\BB_{\cris})$ and $\widetilde{\DD}_{\cris}(V)=(\BB_{\cris}\otimes V)^{G_{\Qp}}$ respectively.
We say that $V$ is crystalline if $V$ has the same $\Qp$-dimension as $\Dcris(V)$. Fix such a $V$. If $j\in\ZZ$, $\Fil^j\Dcris(V)$ denotes the $j$th de Rham filtration of $\Dcris(V)$ given by powers of $t$.

Let $T$ be a lattice of $V$ which is stable under $G_{\Qp}$ and let $A=V/T$. If $K$ is a number field and $v$ is a place of $K$, we define
\[ H^1_f(K_v,V):=\begin{cases}
				\ker\!\left(H^1(K_v,V)\rightarrow H^1(I_v,V)\right)	&\text{if }v\nmid p \\
				\ker\!\left(H^1(K_v,V)\rightarrow H^1(K_v,V\otimes_{\QQ_p}\BB_{\cris})\right)	&\text{if }v\mid p \\
			\end{cases}
\]
where $I_v$ denotes the inertia subgroup at $v$. The local condition $H^1_f(K_v,A)$ is defined as the image of the above under the long exact sequence in cohomology attached to
\[ 0\rightarrow T\rightarrow V\rightarrow A\rightarrow0.
\]
We define the Selmer group of $A$ over $K$ to be
$$
H^1_f(K,A):=\ker\left(H^1(K,A)\rightarrow\prod_v\frac{H^1(K_v,A)}{H^1_f(K_v,A)}\right)
$$
where $v$ runs through the places of $K$.

We write $V(j)$ for the $j$th Tate twist of $V$, i.e.\! $V(j)=V\otimes E e_j$ where $G_{\Qp}$ acts on $e_j$ via $\chi^j$. The Tate dual of $V$ (respectively $T$) is $V^\ast:=\operatorname{Hom}(V,E(1))$ (respectively $T^\ast:=\operatorname{Hom}(V,\calO_E(1))$). We write $A^*=V^*/T^*$.

%+++++++++++++++++++++++++++++++++++++++++++

\subsection{Imaginary quadratic fields}

Let $K$ be an imaginary quadratic field with ring of integers $\OO$. We write $\varepsilon_K$ for the quadratic character associated to $K$, i.e.\! the character on $G_{\QQ}$ which sends $\sigma$ to $1$ if $\sigma\in G_K$ and to $-1$ otherwise.

A Hecke character of $K$ is simply a continuous homomorphism $\psi:\AA_K^\times\rightarrow\CC^\times$ that is trivial on $K^\times$ with complex $L$-function
\[
L(\psi,s)=\prod_v(1-\psi(v)N(v)^{-s})^{-1}
\]
where the product runs through the finite places $v$ of $K$ at which $\psi$ is unramified, $\psi(v)$ is the image of the uniformiser of $K_{v}$ under $\psi$, and $N(v)$ is the norm of $v$. We say that $\psi$ is algebraic of type $(m,n)$ where $m,n\in\ZZ$ if the restriction of $\psi$ to $\AA_{K,\infty}^\times$ is of the form $z\mapsto z^m\bar{z}^n$. 

Fix a finite extension $E$ of $\Qp$ such that $E$ contains the image of $\iota_p\iota_\infty^{-1}\psi|_{\AA^\times_{K,f}}$. We write $V_\psi$ for the one-dimensional $E$-linear representation of $G_K$ normalised so that at finite places $v$ relatively prime to $p$ and the conductor of $\psi$, the action of $\Frob_v$ is via multiplication by $\psi(v)$. We write $\wt{V}_\psi=\IndQK(V_\psi)$.

%+++++++++++++++++++++++++++++++++++++++++++

\subsection{Modular forms}\label{sec:modularforms}

   We will use the term newform to mean a holomorphic modular form of integral weight $\geq2$ which is a normalised cuspidal eigenform for the Hecke algebra generated by all Hecke operators $T_n$ and is an element of the new subspace. Let $f=\sum a_nq^n$ be such a newform of weight $k$, level $\Gamma_1(N)$, and Nebentypus character $\epsilon$. Unless stated otherwise, we assume that $f$ is a CM modular form which is non-ordinary at $p$, i.e. 
   \[
    L(f,s)=\sum a_n(f)n^{-s}=L(\psi,s)
   \]
   for some Hecke character $\psi$ of an imaginary quadratic field $K$ with $p$ inert in $K$. Then, $\psi$ is of type $(k-1,0)$ and $\epsilon_K(p)=-1$. In this case, $a_p(f)$ is always $0$ since $f$ is CM.

We write $E$ for the completion in $\ol{\QQ}_p$ of the image of $\iota_p\iota_\infty^{-1}\psi|_{\AA^\times_{K,f}}$. The coefficient field of $f$ is contained in $E$.

We write $V_f$ for the (dual of the) 2-dimensional $E$-linear representation of $G_{\QQ}$ associated to $f$ inside the cohomology of a Kuga--Sato variety by \cite{deligne69}, so we have a homomorphism $\rho_f:G_\QQ\rightarrow \text{GL}(V_f)$ with $\det(\rho_f)=\epsilon\chi^{k-1}$. Moreover, $V_f\cong\wt{V}_\psi$.

%++++++++++++++++++++++++++++++++++++++++++++++++++++++++++++++
%++++++++++++++++++++++++++++++++++++++++++++++++++++++++++++++

\section{The symmetric power of a modular form}\label{sec:sympower}

Let $f$ be a newform as in Section~\ref{sec:modularforms}. Let $m\ge2$ be an integer. In this section, we collect basic results on the representation $\symV$, as well as the main representation of interest in this article, $V_m$. Throughout this article, we let $r:=\lfloor m/2\rfloor$ and $\wt{r}:=\lceil m/2\rceil$.

\begin{proposition}\label{prop:decomprep}
If $m$ is even, we have a decomposition of $G_\QQ$-representations
\[
\symV\cong\bigoplus_{i=0}^{\wt{r}-1} \Big(\wt{V}_{\psi^{m-2i}}\otimes\left(\varepsilon_K\det \rho_f\right)^i\Big)\oplus\left(\varepsilon_K\det \rho_f\right)^r.
\]
If $m$ is odd, then
\[
\symV\cong\bigoplus_{i=0}^{\wt{r}-1} \Big(\wt{V}_{\psi^{m-2i}}\otimes\left(\varepsilon_K\det \rho_f\right)^i\Big).
\]
\end{proposition}
\begin{proof}
See either \cite[Section~2]{harron12} or \cite[Proposition~5.1]{lei10}.
\end{proof}

\begin{definition}
We define for all $m\ge1$ the representation
\[
V_m=\symV\otimes\det(\rho_f)^{-r}.
\]
\end{definition}

Note that $V_2\cong\ad^0(V_f)$.

\begin{corollary}
As $G_\QQ$-representations, we have an isomorphism
\[
V_m\cong\bigoplus_{i=0}^{\wt{r}-1} \Big(\wt{V}_{\psi^{m-2i}}\otimes\left(\varepsilon_K^i\epsilon^{i-r} \chi^{(i-r)(k-1)}\right)\Big)\oplus\varepsilon_K^r
\]
when $m$ is even. If $m$ is odd,
\[
V_m\cong\bigoplus_{i=0}^{\wt{r}-1} \Big(\wt{V}_{\psi^{m-2i}}\otimes\left(\varepsilon_K^i\epsilon^{i-r}\chi^{(i-r)(k-1)} \right)\Big).
\]
\end{corollary}
\begin{proof}
Recall that $\det(\rho_f)\cong \epsilon\chi^{k-1}$, so this is immediate from Proposition~\ref{prop:decomprep}.
\end{proof}

\begin{proposition}\label{prop:decomporep}
For $0\le i\le\wt{r}-1$, there exist CM newforms $f_i$ of weight $(m-2i)(k-1)+1$ such that
\[
V_m\cong\bigoplus_{i=0}^{\wt{r}-1}\left(V_{f_i}\otimes\chi^{(i-r)(k-1)}\right)\oplus\varepsilon_K^r
\]
when $m$ is even and
\[
V_m\cong\bigoplus_{i=0}^{\wt{r}-1}V_{f_i}\otimes\chi^{(i-r)(k-1)}
\]
when $m$ is odd. Moreover, these $f_i$ have CM by $K$, level prime to $p$, and are non-ordinary at $p$.
\end{proposition}
\begin{proof}
This is discussed in \cite[Proposition 2.2]{harron12}. Since $\psi^{m-2i}$ is a Grossencharacter of type $((m-2i)(k-1),0)$, it corresponds to a CM form $f_i'$ of weight $(m-2i)(k-1)+1$ by \cite[Theorem~3.4]{ribet77}. Therefore, we may take $f_i$ to be the newform whose Hecke eigenvalues agree with those of $f_i'\otimes \varepsilon_K^i\epsilon^{i-r}$ outside of the level of $f_i^\prime$ and the conductor of $\varepsilon_K^i\epsilon^{i-r}$. It is CM because it corresponds to the Grossencharacter $\psi^{m-2i}\otimes \varepsilon_K^i\epsilon^{i-r}$.

Since $p$ does not divide the conductor of $\psi^{m-2i}$, we see that the level of $f_i'$ is not divisible by $p$. The conductor of $\varepsilon_K^i\epsilon^{i-r}$ is coprime to $p$, too, hence the level of $f_i$ is also prime to $p$. and $a_p(f_i')=0$. Since $p$ is inert in $K$, $f_i$ is non-ordinary at $p$.
\end{proof}

\begin{corollary}\label{cor:decompoL}
The complex $L$-function admits a decomposition
\[
L(V_m,s)=\left(\prod_{i=0}^{\wt{r}-1} L\left(f_i,s+(r-i)(k-1)\right)\right)\cdot
	\begin{cases}
		L(\varepsilon_K^r,s)		& m\text{ even,} \\
		1 & m\text{ odd.}
\end{cases}
\]
\end{corollary}

Now, we discuss the Hodge theory of $V_m$ at $\infty$ and describe its critical twists. Recall that the Hodge structure of $V_f$ is of the form $H^{0,k-1}\oplus H^{k-1,0}$, with each summand one-dimensional. Taking symmetric powers, we see that the Hodge structure of $\symV$ is of the form
\[ \bigoplus_{a=0}^mH^{a(k-1),(m-a)(k-1)}
\]
with each summand one-dimensional. In particular, if $m$ is even, there is a ``middle'' term $H^{r(k-1),r(k-1)}$ and to specify the Hodge structure we must also say how the ``complex conjugation'' involution, denoted $F_\infty$, acts on this term: it acts trivially. When we twist by $\det(\rho_f)^{-r}$, this shifts the indices by $-r(k-1)$ and, when $m$ is even, twists the action of $F_\infty$ on the middle term by
\begin{align*}
	\det(\rho_f)^{-r}(\Frob_\infty)	&=\left(\epsilon(-1)(-1)^{(k-1)}\right)^{-r}\\
							&=\left((-1)^k(-1)^{(k-1)}\right)^{r}\\
							&=(-1)^r.
\end{align*}
We have thus proved the following.
\begin{lemma}
	The Hodge structure of $V_m$ is
	\[ \bigoplus_{a=0}^mH^{-(r-a)(k-1),(\wt{r}-a)(k-1)}
	\]
	with each summand one-dimensional. If $m$ is even, then $F_\infty$ acts on $H^{0,0}$ by $(-1)^r$.
\end{lemma}
Therefore, if $d^\pm$ denotes the dimension of the part of the Hodge structure of $V_m$ on which $F_\infty$ acts by $\pm1$, then
\begin{align}
	d^+&=\begin{cases}
			r	&\text{if }m\equiv2\mod{4} \\
			r+1	&\text{otherwise,}
		\end{cases}
	\\
	d^-&=\begin{cases}
			r	&\text{if }m\equiv0\mod{4} \\
			r+1	&\text{otherwise.}
		\end{cases}
\end{align}

From this, we can determine the critical twists of $V_m$ (in the sense of \cite{deligne79}).
\begin{lemma}
	Let $C_m$ denote the set of all pairs $(\theta,j)$ such that $V_m\otimes\theta\chi^j$ is critical, where $\theta$ is a Dirichlet character of $p$-power conductor. Then,
	\[ C_m=\begin{cases}
			\left\{(\theta,j):-(k-1)+1\leq j\leq(k-1)\text{ and }\theta\chi^j(-1)=\sgn\!\left(j-\frac{1}{2}\right)(-1)^r\right\}	&\text{if }m\text{ is even}\\
			\left\{(\theta,j):1\leq j\leq k-1\right\}			&\text{if }m\text{ is odd.}
		\end{cases}
	\]
\end{lemma}
\begin{proof}
	The integers $j$ are forced to lie between
	\[ \max_{a<b}\{a:H^{a,b}\neq0\}+1\text{ and }\min_{a>b}\{a:H^{a,b}\neq0\}.
	\]
	When there is a non-zero $H^{a,b}$ with $a=b$, there is an additional condition depending on the parity of $\theta\chi^j$ and the sign of $j-\frac{1}{2}$ (i.e.\ whether $j$ is positive or non-positive) which is determined by the action of $F_\infty$. This occurs when $m$ is even. We leave the details to the reader.
\end{proof}
We remark that this lemma holds without the requirement that $\theta$ have $p$-power conductor.

Recall that the Hodge polygon $P_H(x,V_m)$ (viewed as a function of $x\in[0,\dim(V_m)]$) is the increasing piecewise-linear function whose segment of slope $a$ has horizontal length $\dim H^{a,b}$. The vertices of the Hodge polygon of $V_m$ are as follows.
\begin{lemma}\label{lem:HodgePoly}
	For $a=0,1,\dots,m$,
	\[ P_H(a+1,V_m)=(k-1)\left(\frac{(a+1)(a-2r)}{2}\right).
	\]
\end{lemma}
\begin{proof}
	Since each piece of the Hodge decomposition of $V_m$ has dimension one, the $y$-coordinate at the end of the $a$th segment is
	\[ \sum_{c=0}^a-(r-c)(k-1).
	\]
	The result follows from summing this.
\end{proof}

This whole discussion of the Hodge theory at $\infty$ holds for any newform (as we have used the term).

We will also require knowledge of the Hodge theory of $V_m$ at $p$.

\begin{lemma}
	If $m$ is even, the eigenvalues of $\vp$ on $\Dcris(V_m)$ are $\pm1$ with multiplicity $d^\pm$. When $m$ is odd, the eigenvalues are $\pm\alpha$ each with multiplicity $d^\pm=\wt{r}$, where $\pm\alpha$ are the roots of $x^2+\epsilon(p)p^{k-1}$.
\end{lemma}
\begin{proof}
	Let $\alpha$ and $\ol{\alpha}$ be the roots of $x^2+\epsilon(p)p^{k-1}$, so that $\ol{\alpha}=-\alpha$. It follows from work of Saito (\cite{saito97}) that $\alpha$ and $\ol{\alpha}$ are the eigenvalues of $\vp$ on $\Dcris(V_f)$. Therefore, there is a basis of $\Dcris(\symV)$ in which $\vp$ is diagonal with entries $\alpha^a\ol{\alpha}^{m-a}$ for $a=0,1,\dots,m$. Twisting by $\det(\rho_f)^{-r}$ translates into multiplying this matrix by $(\alpha\ol{\alpha})^{-r}$. If $m$ is even, the diagonal entries become $\left(\frac{\alpha}{\ol{\alpha}}\right)^a$ for $a=-\wt{r},-r+1,\dots,r$. For $m$ odd, the diagonal entries become $\alpha\left(\frac{\alpha}{\ol{\alpha}}\right)^a$ for $a=-\wt{r},-r,\dots,r$. The result follows.
\end{proof}
We remark that this lemma holds without the assumption that $f$ is CM as long as $a_p(f)=0$.

It is not necessary, but convenient, to use the decomposition of Proposition \ref{prop:decomporep} in order to determine the filtration on $\Dcris(V_m)$. We obtain the following description of $\Dcris(V_m)$.
\begin{lemma}\label{lem:DcrisVm}
	If $m$ is even, there is a basis $v,v_0,\ol{v}_0,v_1,\ol{v}_1,\dots,v_{\wt{r}-1},\ol{v}_{\wt{r}-1}$ of $\Dcris(V_m)$ such that
	\[\begin{array}{lcl}
		\vp(v)&=&(-1)^rv, \\
		\vp(v_j)&=&v_j, \\
		\vp(\ol{v}_j)&=&-\ol{v}_j
	\end{array}
	\]
	and, for $j=-r,-r+1,\dots,\wt{r}$, the graded piece $\gr^{j(k-1)}\Dcris(V_m)$ is generated by the image of
	\[ \begin{cases}
												v_{r+j}		&\text{if }j<0, \\
												v					&\text{if }j=0, \\
												v_{\wt{r}-j}+\ol{v}_{\wt{r}-j}		&\text{if }j>0.
											\end{cases}
	\]
	If $m$ is odd, there is a basis $v_0,\ol{v}_0,v_1,\ol{v}_1,\dots,v_{\wt{r}-1},\ol{v}_{\wt{r}-1}$ of $\Dcris(V_m)$ such that
	\[\begin{array}{lcl}
		\vp(v_j)&=&v_j, \\
		\vp(\ol{v}_j)&=&-\ol{v}_j
	\end{array}
	\]
	and, for $j=-r,-r+1,\dots,\wt{r}$, the graded piece $\gr^{j(k-1)}\Dcris(V_m)$ is generated by the image of
	\[ \begin{cases}
												v_{r+j}		&\text{if }j\leq0, \\
												v_{\wt{r}-j}+\ol{v}_{\wt{r}-j}		&\text{if }j>0.
											\end{cases}
	\]
\end{lemma}
\begin{proof}
	Suppose $m$ is even. The crystalline Dieudonn\'e module of a Dirichlet character $\eta$ (of prime-to-$p$ conductor) has Hodge--Tate weight $0$ and $\vp$ acts by $\eta(p)$. Let $v$ be a basis of $\Dcris(\epsilon_K^r)$. For $i=0,1,\dots,\wt{r}-1$, let $v_i^\prime,\ol{v}_i^\prime$ be a basis of $\Dcris(V_{f_i})$ in which the action of $\vp$ is diagonal. Since $f_i$ has weight $2(r-i)(k-1)+1$, its crystalline Dieudonn\'e module $\Dcris(V_{f_i})$ has non-zero graded pieces in degrees $0$ and $2(r-i)(k-1)$. Furthermore, its filtration is determined by the degree $2(r-i)(k-1)$ piece. In order for the Dieudonn\'e module to be weakly admissible, this piece cannot contain either $v_i^\prime$ or $\ol{v}_i^\prime$; therefore, up to rescaling,
	\[ \Fil^{2(r-i)(k-1)}\Dcris(V_{f_i})=\langle v_i^\prime+\ol{v}_i^\prime\rangle.
	\]
	Twisting by $-(r-i)(k-1)$ shifts the filtration so that it drops into degrees $\pm(r-i)(k-1)$. Let $v_i,\ol{v}_i$ denote the twisted basis. Given what we know from lemma \ref{lem:DcrisVm}, $\vp(v_i)=\pm v_i$ and $\vp(\ol{v}_i)=\pm\ol{v}_i$. Up to reordering, $\vp(v_i)=v_i$ and $\vp(\ol{v}_i)=-\ol{v}_i$, since the Frobenius eigenvalues of $V_{f_i}$ are negatives of each other (since $a_p(f_i)=0$). The result follows from the decomposition of Proposition \ref{prop:decomporep}. The proof when $m$ is odd is along the same lines and we leave it to the reader.
\end{proof}
Therefore, the dimension of the de Rham tangent space $t_{\dR}(V_m):=\Dcris(V_m)/\Fil^0\Dcris(V_m)$ is $d^+$, unless $4|m$, in which case it is $d^-$. Recall that the Newton polygon $P_N(x,V_m)$ is the piecewise-linear increasing function whose segment of slope $\lambda$ has length equal to the number of eigenvalues of $\vp$ whose valuation is $\lambda$. The following lemma follows immediately from the previous lemma.
\begin{lemma}\label{lem:NewtonPoly}
	If $m$ is even, the Newton polygon of $V_m$ is the horizontal line from $(0,0)$ to $(m+1,0)$. When $m$ is even, it is the straight line from $(0,0)$ to $\left(m+1,\frac{(k-1)(m+1)}{2}\right)$.
\end{lemma}

%++++++++++++++++++++++++++++++++++++++++++++++++++++++++++++++
%++++++++++++++++++++++++++++++++++++++++++++++++++++++++++++++

\section{Plus and minus \texorpdfstring{$p$-adic $L$-functions}{p-adic L-functions}}\label{sec:plusminus}

\subsection{Some known \texorpdfstring{$p$-adic $L$-functions}{p-adic L-functions}}

In this section, we review some facts concerning the $p$-adic $L$-functions of Dirichlet characters and newforms, as well as Pollack's plus and minus $p$-adic $L$-functions of newforms.

The following is a consequence of classical results of Kubota--Leopoldt and Iwasawa (\cite{kubotaleopoldt,iwasawa72}).
\begin{theorem}\label{thm:KL}
If $\eta$ is a non-trivial Dirichlet character of conductor prime to $p$ (resp. the trivial character $\triv$) and $E/\QQ_p$ contains the values of $\eta$, then there exists a unique $L_\eta\in\Lambda_{\mathcal{O}_E}(G_\infty)$ (resp. $L_\triv$ with $(\gamma_0-1)(\gamma_0-\chi(\gamma_0))L_\triv\in\Lambda_{\mathcal{O}_E}(G_\infty)$) such that 
\[
\theta\chi^j(L_\eta)= e_{\eta}(\theta,j)\frac{L(\eta\theta^{-1},j)}{\Omega_\eta(\theta,j)}\\
\]
for any integer $j$ and Dirichlet character $\theta$ of conductor $p^n$ satisfying
\[ \theta\chi^j(-1)=\sgn\!\left(j-\frac{1}{2}\right)\eta(-1)
\]
where
\[
e_\eta(\theta,j)=\left(\frac{p^j}{\eta(p)}\right)^n\!\cdot\begin{cases}
				\left(1-p^{j-1}(\eta^{-1}\theta)(p)\right)	& \text{if }\theta\chi^j(-1)=\eta(-1), j\geq1,\\
				\left(1-p^{-j}(\eta\theta^{-1})(p)\right)	& \text{if }\theta\chi^j(-1)=-\eta(-1), j\leq0,
			\end{cases}
\]
and
\[
	\frac{1}{\Omega_\eta(\theta,j)}=\begin{cases}
								\displaystyle\frac{2\Gamma(j)\mathfrak{f}_\eta^j}{\tau(\theta^{-1})(-2\pi i)^j\theta(\mf{f}_\eta)\tau(\eta)}	& \text{if }\theta\chi^j(-1)=\eta(-1), j\geq1,\\
								\displaystyle\left(\frac{p^j}{\eta(p)}\right)^{-n}								& \text{if }\theta\chi^j(-1)=-\eta(-1), j\leq0.
							\end{cases}
\]
\end{theorem}
\begin{proof}
	Let us briefly sketch the construction of these elements. The work of Iwasawa in \cite[Chapter 6]{iwasawa72}, building upon the main result of \cite{kubotaleopoldt}, attaches to a non-trivial \textit{even} Dirichlet character $\xi$ of conductor not divisible by $p^2$ an Iwasawa function $L_\xi^{\KL}\in\Lambda_{\mathcal{O}_E}(\Gamma)$ (for any $E$ containing the values of $\xi$) such that
	\[
		(\chi/\omega)^n(L_\xi^{\KL})=\left(1-(\xi\omega^{n-1})(p)p^{-n}\right)L(n,\xi\omega^{n-1})
	\]
	for all $n\in\ZZ_{\leq0}$. When $\xi=\triv$, Iwasawa constructs $L_\triv^{\KL}\in\Frac(\Lambda_{\ZZ_p}(G_\infty))$ such that $(\gamma_0-\chi(\gamma_0))L_\triv^{\KL}\in\Lambda_{\ZZ_p}(G_\infty)$ with the same interpolation property. We write
	\[ L_\eta=\sum_{a\in\ZZ/(p-1)\ZZ}\pi_{\omega^a}L_{\eta,a},
	\]
	where $\pi_{\omega^a}$ are the idempotents defined in \S\ref{sec:iwasawaalgebras}. Then, we take
	\[ L_{\eta,a}:=\begin{cases}
						\Tw_{-1}\!\left(\left(L_{\eta^{-1}\omega^a}^{\KL}\right)^\iota\right)	& \text{if }\eta(-1)=(-1)^a, \\
						L_{\eta\omega^{1-a}}^{\KL}		& \text{if }\eta(-1)=-(-1)^a.
					\end{cases}
	\]
\end{proof}

Next, we recall the existence of $p$-adic $L$-functions for newforms.
\begin{theorem}\label{thm:AV}
Let $f$ be a newform of level prime to $p$. Let $\alpha$ be a root of $x^2-a_p(f)x+\epsilon(p)p^{k-1} $such that $h:=\ord_p(\alpha)<k-1$ and let $\ol{\alpha}$ denote the other root. There exist $\Omega_f^\pm\in\CC^\times$ and $L_{f,\alpha}\in\HH_{h,E}(G_\infty)$ such that for all integers $j\in[1,k-1]$ and Dirichlet characters $\theta$ of conductor $p^n$, we have 
\[
\theta\chi^j(L_{f,\alpha})=e_{f,\alpha}(\theta,j)\frac{L(f,\theta^{-1},j)}{\Omega_f(\theta,j)}
\]
where
\[
e_{f,\alpha}(\theta,j)=\left(\frac{p^j}{\alpha}\right)^n\left(1-\theta^{-1}(p)\ol{\alpha}p^{-j}\right)\left(1-\frac{\theta(p)p^{j-1}}{\alpha}\right),
\]
and
\[
	\frac{1}{\Omega_f(\theta,j)}=\frac{\Gamma(j)}{\tau(\theta^{-1})(-2\pi i)^{j}\Omega_f^\delta},
\]
where $\delta=(-1)^j\theta(-1)$. Moreover, $L_{f,\alpha}$ is uniquely determined by its values at $\theta\chi^j$ given above.
\end{theorem}
\begin{proof}
This is the main result of \cite{amicevelu}.
\end{proof}

Now, we recall Pollack's construction of plus and minus $p$-adic $L$-functions in \cite{pollack03}. The plus and minus logarithms are defined as follows.

\begin{definition}
Let $b\ge1$ be an integer and define\footnote{We take a different normalization than Pollack does. In \cite[Corollary 4.2]{pollack03}, he attaches to an $f$ of weight $k$ two elements that we denote here $\log^\pm_{\text{Pollack}}$ which are related to our logarithms by
\[ \log^\pm_{k-1}=p^{k-1}\Tw_{-1}\log^\pm_{\text{Pollack}}.
\]}
\begin{align*}
\log^+_b&=\prod_{a=1}^{b}\prod_{n=1}^\infty\frac{\Phi_{2n}(\chi(\gamma_0)^{-a}\gamma_0)}{p},\\
\log^-_b&=\prod_{a=1}^{b}\prod_{n=1}^\infty\frac{\Phi_{2n-1}(\chi(\gamma_0)^{-a}\gamma_0)}{p},
\end{align*}
where $\Phi_m$ denotes the $p^m$th cyclotomic polynomial.
\end{definition}
Recall that $\log^\pm_b\in\HH_{\frac{b}{2}}(G_\infty)$ and that the zeroes of $\log^+_b$ (resp., $\log^-_b$) are simple and located at $\chi^j\theta$ for all integers $j\in[1,b]$ and for all $\theta$ of conductor $p^{2n+1}$ (resp., $p^{2n}$) for positive integers $n$.

The following is the main result of \cite{pollack03}.
\begin{theorem}\label{thm:pollack}
Let $f$ and $\alpha$ be as in Theorem~\ref{thm:AV} with $a_p(f)=0$. There exist $L_f^\pm\in\Lambda_E(G_\infty)$ (independent of $\alpha$) such that
\[
L_{f,\alpha}=L_f^+\log_{k-1}^++\alpha L_f^-\log_{k-1}^-.
\]
\end{theorem}

In particular, if $\alpha$ and $\ol{\alpha}$ are the two roots to $x^2+\epsilon(p)p^{k-1}=0$, we have $\ol{\alpha}=-\alpha$ and
\begin{equation}\label{eqn:Lf+-}
\begin{array}{ll}
	L_{f}^+&\displaystyle=\frac{\ol{\alpha}L_{f,\alpha}-\alpha L_{f,\ol{\alpha}}}{(\ol{\alpha}-\alpha)\log_{k-1}^+},\\
	L_{f}^-&\displaystyle=\frac{L_{f,\ol{\alpha}}-L_{f,\alpha}}{(\ol{\alpha}-\alpha)\log_{k-1}^-}.
\end{array}
\end{equation}
Therefore, we can readily combine Theorems~\ref{thm:AV} and~\ref{thm:pollack} to obtain the interpolating formulae of $L_{f}^\pm$ as follows.

\begin{lemma}\label{lem:inter}
Let $j\in[1,k-1]$ be an integer and $\theta$ a Dirichlet character of conductor $p^n>1$. Write $\delta$ as in Theorem~\ref{thm:AV}. If $n$ is even, then
\[
\theta\chi^j(L_f^+)=\frac{1}{(-\epsilon(p))^{n/2}p^{n(\frac{k-1}{2}-j)}\theta\chi^j(\log_{k-1}^+)}\cdot\frac{L(f,\theta^{-1},j)}{\Omega_f(\theta,j)},
\]
and if $n=1$,
\[ \theta\chi^j(L_f^+)=0.
\]
For all odd $n$, we have
\[
\theta\chi^j(L_f^-)=\frac{1}{(-\epsilon(p))^{(n+1)/2}p^{n(\frac{k-1}{2}-j)+\frac{k-1}{2}}\theta\chi^j(\log_{k-1}^-)}\cdot\frac{L(f,\theta^{-1},j)}{\Omega_f(\theta,j)}.
\]
Moreover,
\begin{align*}
\chi^j(L_f^+)&=\frac{1-p^{-1}}{\chi^j(\log_{k-1}^+)}\cdot\frac{L(f,j)}{\Omega_f(\mathbf{1},j)},\\
\chi^j(L_f^-)&=\frac{p^{-j}+\epsilon(p)^{-1}p^{j-k}}{\chi^j(\log_{k-1}^-)}\cdot\frac{L(f,j)}{\Omega_f(\mathbf{1},j)}.
\end{align*}
\end{lemma}
\begin{proof}

	As a sample calculation, for $L_f^-$ we have
	\begin{align*}
		\theta\chi^j(L_f^-)&=\left(e_{f,\ol{\alpha}}(\theta,j)-e_{f,\alpha}(\theta,j)\right)\cdot\frac{L(f,\theta^{-1},j)}{(\ol{\alpha}-\alpha)\theta\chi^j(\log_{k-1}^-)\Omega_f(\theta,j)} \\
			&=\left(\left(\frac{p^j}{\ol{\alpha}}\right)^n-\left(\frac{p^j}{\alpha}\right)^n\right)\cdot\frac{L(f,\theta^{-1},j)}{(-2\alpha)\theta\chi^j(\log_{k-1}^-)\Omega_f(\theta,j)} \\
			&=\left(\frac{2p^{nj}}{(-\alpha)^n}\right)\cdot\frac{L(f,\theta^{-1},j)}{(-2\alpha)\theta\chi^j(\log_{k-1}^-)\Omega_f(\theta,j)} \\
			&=\left(\frac{p^{nj}}{(-\alpha)^{n+1}}\right)\cdot\frac{L(f,\theta^{-1},j)}{\theta\chi^j(\log_{k-1}^-)\Omega_f(\theta,j)} \\
			&=\left(\frac{p^{nj}}{(-\alpha\ol{\alpha})^{\frac{n+1}{2}}}\right)\cdot\frac{L(f,\theta^{-1},j)}{\theta\chi^j(\log_{k-1}^-)\Omega_f(\theta,j)} \\
			&=\left(\frac{p^{nj}}{(-\alpha\ol{\alpha})^{\frac{n+1}{2}}}\right)\cdot\frac{L(f,\theta^{-1},j)}{\theta\chi^j(\log_{k-1}^-)\Omega_f(\theta,j)},
	\end{align*}
	which gives the result.
\end{proof}

\begin{remark}\mbox{}
\begin{enumerate}
	\item For a fixed $j$, $L_f^\pm$ is uniquely determined by $\theta\chi^j(L_f^+)$ for an infinite number of $\theta$ since it is in $\Lambda_E(G_\infty)$. Therefore, it is uniquely determined by the interpolating properties given in Lemma~\ref{lem:inter}.
	\item The interpolation property for the other $n$ is not known because $\log^\pm_{k-1}$ vanishes there. To distinguish the above characters, $\theta\chi^j$ will be called a \emph{plus-critical twist} (resp., a \emph{minus-critical twist}) of $f$ if $1\leq j\leq k-1$ with $\theta$ of conductor $p^n\geq1$ with $n=1$ or even (resp., with $n=0$ or odd).
	\item It is clear from the lemma that $L_f^+$ vanishes at the plus-critical twist $\theta\chi^j$ if and only if $\theta$ has conductor $p$ or the complex $L$-value vanishes at the central point. Furthermore, $L_f^-$ vanishes at the minus-critical twist $\theta\chi^j$ if and only if $k$ is even, $\epsilon(p)=-1$, $\theta=\mathbf{1}$, and $j=\frac{k}{2}$ (the central point) or the complex $L$-value vanishes at the central point.
\end{enumerate}
\end{remark}

For notational simplicity later on, we define the following.
\begin{definition}
Let $j\in[1,k-1]$ be an integer and $\theta$ a Dirichlet character of conductor $p^n$ with $n\geq0$. We define
\[
e_f^+(\theta,j)=
\begin{cases}
\displaystyle\frac{1-\theta(p)p^{-1}}{(-\epsilon(p))^{n/2}p^{n(\frac{k-1}{2}-j)}\theta\chi^j(\log_{k-1}^+)}&\text{if }n\text{ is even,}\\
0&\text{if $n=1$}.
\end{cases}
\]
Similarly, we define
\[
e_f^-(\theta,j)=\begin{cases}
\displaystyle\frac{1}{(-\epsilon(p))^{(n+1)/2}p^{n(\frac{k-1}{2}-j)+\frac{k-1}{2}}\theta\chi^j(\log_{k-1}^-)}&\text{if }n\text{ is odd,}\\
\displaystyle\frac{p^{-j}+\epsilon(p)^{-1}p^{j-k}}{\chi^j(\log_{k-1}^-)}&\text{if }n=0.
\end{cases}
\]
\end{definition}

%++++++++++++++++++++++++++++++++++++++++++++++++++++++++++++++

\subsection{Mixed plus and minus \texorpdfstring{$p$-adic $L$-functions}{p-adic L-functions} for \texorpdfstring{$V_m$}{Vm}}\label{sec:mixed_plus_minus}

We now define mixed plus and minus $p$-adic $L$-functions for $V_m$ with the decomposition of Corollary \ref{cor:decompoL} in mind.

Let $\mf{S}$ denote the set of all choices of $\pm$ for $i=0,1,\dots,\wt{r}-1$, so that $\#\mf{S}=2^{\wt{r}}$. For $\mf{s}\in\mf{S}$, let $\mathfrak{s}_i$ denote the $i$th choice and let $\mf{s}^\pm$ denote the number of plusses and minuses, respectively.

\begin{definition}\label{def:mixed_plus_minus}
Let $f_i$ be the modular forms as given in Proposition~\ref{prop:decomporep}. For an $\mf{s}\in\mf{S}$, define
\[
L_{V_m}^\mf{s}=\left(\prod_{i=0}^{\wt{r}-1}\Tw_{(r-i)(k-1)}\left(L_{f_i}^{\mf{s}_i}\right)\right)\cdot
	\begin{cases}
		L_{\varepsilon_K^r}	& m\text{ even,} \\
		1				& m\text{ odd.}
	\end{cases}
\]
\end{definition}

\begin{remark}
We have that $L_{V_m}^{\mf{s}}\in\Lambda_E(G_\infty)$ for $m$ odd; whereas if $m$ is even, then $L_{V_m}^{\mf{s}}\in\Lambda_E(G_\infty)$ unless $4|m$, in which case $(\gamma_0-1)(\gamma_0-\chi(\gamma_0))L_{V_m}^{\mf{s}}\in\Lambda_E(G_\infty)$.
\end{remark}

\begin{definition}\label{def:deltacritical}
	If $\theta$ is a Dirichlet character of conductor $p^n$, $j\in\ZZ$, and $\mf{s}\in\mf{S}$, then the character $\theta\chi^j$ will be called $\mf{s}$-critical (for $V_m$) if $(\theta,j)\in C_m$ and
	\[	\begin{cases}
			n=1\text{ or is even}		& \text{if }\mf{s}^-=0, \\
			n=0\text{ or is odd}		& \text{if }\mf{s}^+=0, \\
			n=0\text{ or }1			& \text{otherwise}.
		\end{cases}
	\]

	For $\mf{s}\in\mf{S}$ and for $\mf{s}$-critical $\theta\chi^j$, define
	\[ e^\mf{s}_{V_m}(\theta,j):=\left(\prod_{i=0}^{\wt{r}-1}e^{\mf{s}_i}_{f_i}(\theta,j+(r-i)(k-1))\right)\cdot
		\begin{cases}
			e_{\epsilon_K^r}(\theta,j)	&m\text{ even,} \\
			1					&m\text{ odd.}
		\end{cases}
	\]

	For all $(\theta,j)\in C_m$, let
	\[
		\Omega_m(\theta,j):=\left(\prod_{i=0}^{\wt{r}-1}\Omega_{f_i}(\theta,j+(r-i)(k-1))\right)\cdot
			\begin{cases}
				\Omega_{\epsilon_K^r}(\theta,j)	&m\text{ even,} \\
				1							&m\text{ odd.}
			\end{cases}
	\]
\end{definition}

\begin{theorem}\label{thm:inter}
Let $\mf{s}\in\mf{S}$. If $\theta\chi^j$ is $\mf{s}$-critical, then
\[
\theta\chi^j\left(L_{V_m}^\mf{s}\right)=e^\mf{s}_{V_m}(\theta,j)\frac{L(V_m,\theta^{-1},j)}{\Omega_m(\theta,j)}.
\]
\end{theorem}
\begin{proof}
This follows from the definitions together with Theorem \ref{thm:KL} and Lemma \ref{lem:inter}.
\end{proof}

Let $\omega^a$ be a character on $\Delta$. If $L_{V_m}^\mf{s}=\sum_{\sigma\in\Delta,n\ge0}a_{\sigma,n}\sigma(\gamma_0-1)^n$, the $\omega^a$-isotypical component of $L_{V_m}^\mf{s}$ is given by
\[
 L_{V_m}^{\mf{s},\omega^a}:=\pi_{\omega^a}\cdot L_{V_m}^\mf{s}=\pi_{\omega^a}\sum_{n\ge0}\left(\sum_{\sigma\in\Delta}a_{\sigma,n}\eta(\sigma)\right)(\gamma_0-1)^n.
\]
Therefore, 
\[
\theta\chi^j\left(L_{V_m}^{\mf{s},\omega^a}\right)=
\begin{cases}
\theta\chi^j\left( L_{V_m}^\mf{s}\right)&\text{if $\theta\chi^j|_\Delta=\omega^a$,}\\
0&\text{otherwise.}
\end{cases}
\]

\begin{proposition}\label{prop:nonzero}
Let $\omega^a$ be a character on $\Delta$. We have $L_{V_m}^{\mf{s},\omega^a}\ne 0$. Moreover, if either of $\mf{s}^+$ or $\mf{s}^-$ vanishes, it is uniquely determined by its interpolation property given by Theorem~\ref{thm:inter}.
\end{proposition}
\begin{proof}
The proof of \cite[Lemma~6.5]{lei09} shows that $L_{f_i}^{\pm,\omega^a}\ne0$ for all $i$. Furthermore, it is known that $L_{\epsilon_K^r}^{\omega^a}\neq0$. The same then follows for $L_{V_m}^{\mf{s},\omega^a}$ by definition. The uniqueness is immediate since $L_{V_m}^{\mf{s},\omega^a}\in \pi_{\omega^a}\Lambda_E(G_\infty)$ (or possibly its field of fractions) and when $\mf{s}$ consists of only plusses or minuses, there is at least one $j$ for which there are infinitely many $\theta$ such that $\theta\chi^j$ is $\mf{s}$-critical.
\end{proof}

%++++++++++++++++++++++++++++++++++++++++++++++++++++++++++++++
%++++++++++++++++++++++++++++++++++++++++++++++++++++++++++++++

\section{Selmer groups and the Main Conjecture of Iwasawa Theory}\label{sec:MC}
In this section, we formulate the Main Conjecture of Iwasawa Theory for $V_m$. Since we are in the non-ordinary setting, this involves a relation between the mixed plus and minus $p$-adic $L$-functions introduced in the previous section and the mixed plus and minus Selmer groups that we define below. We once again take advantage of the decomposition of Proposition~\ref{prop:decomporep} to define these Selmer groups of $V_m$ as the direct sum of the plus and minus Selmer groups of the $\epsilon_K^r$ and $V_{f_i}$. The latter were defined by the second author in \cite[Section~4]{lei09}. We then apply the results of Mazur--Wiles \cite{mazurwiles} and \cite[Section~7]{lei09}.

%++++++++++++++++++++++++++++++++++++++++++++++++++++++++++++++

\subsection{Selmer groups}
Fix a $G_\QQ$-stable $\calO_E$-lattice $T_{\psi}$ of $V_\psi$. It gives rise to natural $\calO_E$-lattices $T_{\psi^{m-2i}}$ for $0\le i\le \wt{r}-1$, which in turn give rise to natural $\calO_E$-lattices $T_{f_i}$ in $V_{f_i}$ which are stable under $G_{\QQ}$. Write $V_{\epsilon_K^{r}}$ (resp. $T_{\epsilon_K^{r}}$) for the representation over $E$ (resp. $\calO_E$) associated to $\epsilon_K^{r}$ when $m$ is even. Let $A_{f_i}=V_{f_i}/T_{f_i}$, and $A_{\epsilon_K^{m/2}}=V_{\epsilon_K^{m/2}}/T_{\epsilon_K^{m/2}}$. We briefly remind the reader about the Selmer groups of $A_\triv^\ast$, $A_{\epsilon_K}^\ast$, and $A_{f_i}^\ast$.

Let $\theta$ denote $\triv$ or $\epsilon_K$ and let $F=\QQ$ or $K$, respectively. Classical Iwasawa theory studies the module $X_\infty:=\Gal(M_\infty/F_\infty)$, where $M_\infty$ is the maximal $p$-abelian extension of $F_\infty$ unramified outside of $p$. The arguments of \cite[Section~1]{G89} show that
\begin{equation}\label{eqn:thetaSelmerFinfty}
	H^1_f(F_\infty,A_\theta^\ast)\cong\Hom(X_\infty,A_\theta^\ast).
\end{equation}
In order to get the Selmer group over $k_\infty$ when $\theta=\epsilon_K$, we take the $\epsilon_K$-isotypic piece. Thus,
\begin{equation}\label{eqn:thetaSelmer}
	H^1_f(k_\infty,A_\theta^\ast)\cong\Hom(X_\infty,A_\theta^\ast)^\theta.
\end{equation}

We now make an additional hypothesis.
\begin{hypothesis}\label{hyp2}
For $0\le i\le \wt{r}-1$, we have $(p+1)\nmid (m-2i)(k-1)$.
\end{hypothesis}

\begin{lemma}
If Hypothesis~\ref{hyp2} holds, then $\left(A_{f_i}^\ast(j)\right)^{G_{\Qpn}}=0$ for all $0\le i\le \wt{r}-1$, $j\in\ZZ$ and $n\in\ZZ_{\ge0}$.
\end{lemma}
\begin{proof}
Recall that $f_i$ is of weight $(m-2i)(k-1)+1$, so Hypothesis~\ref{hyp2} implies \cite[Assumption~(2)]{lei09}. Therefore, the result follows from \cite[Lemma~4.4]{lei09}.
\end{proof}

As in \cite[Section~4]{lei09}, the restriction map $H^1(\QQ_{p,\nu},T_{f_i}^\ast)\rightarrow H^1(\Qpn,T_{f_i}^\ast)$ is injective for any integers $n\ge \nu\ge 0$. On identifying the former as a subgroup of the latter, we define 
\begin{equation}\label{eq:jump}
H^{1,\pm}_f(\Qpn,T_{f_i}^\ast)=\left\{x\in H^1_f(\Qpn,T_{f_i}^\ast):\cor_{n/\nu+1}(x)\in H^1_f(\QQ_{p,\nu},T_{f_i}^\ast)\text{ for all }\nu\in S_n^\pm \right\}
\end{equation}
where $\cor$ denotes the corestriction map and
\begin{eqnarray*}
S_n^+&=&\{\nu\in[0,n-1]:\nu\text{ even}\},\\
S_n^-&=&\{\nu\in[0,n-1]:\nu\text{ odd}\}.
\end{eqnarray*}
Let 
\[
H^{1,\pm}_f(\Qpn,A_{f_i}^\ast)=H^{1,\pm}_f(\Qpn,T_{f_i}^\ast)\otimes E/\calO_E\subset H^1_f(\Qpn,A_{f_i}^\ast).
\]
The plus and minus Selmer groups over $k_n=\QQ(\mu_{p^n})$ are defined by
\[
H^{1,\pm}_f(k_n,A_{f_i}^\ast)=\ker\left(H^1_f(k_n,A_{f_i}^\ast)\rightarrow\frac{H^1(\Qpn,A_{f_i}^\ast)}{H^{1,\pm}_f(\Qpn,A_{f_i}^\ast)}\right)
\]
and those over $k_\infty=\QQ(\mu_{p^\infty})$ are defined by
\[
H^{1,\pm}_f(k_\infty,A_{f_i}^\ast)=\lim_{\longrightarrow}H^{1,\pm}_f(k_n,A_{f_i}^\ast),
\]
which can be identified as a subgroup of $H^1(k_\infty,A_{f_i}^\ast)$. For $j=0,1,\dots,(m-2i)(k-1)-1$, $\chi^j$ gives an isomorphism
\[ H^{1,\pm}_f(k_\infty,A_{f_i}^\ast)\cong H^{1,\pm}_f(k_\infty,A_{f_i}^\ast(j)).
\]
We use this to define the Tate twisted plus and minus Selmer groups of $A_{f_i}^\ast(j)$ by
\[ H^{1,\pm}_f(k_\infty,A_{f_i}^\ast(j)):=H^{1,\pm}_f(k_\infty,A_{f_i}^\ast)\otimes\chi^j.
\]

Now,
\[ T_m:=\left(\bigoplus_{i=0}^{\wt{r}-1}T_{f_i}((i-r)(k-1))\right)\oplus \begin{cases}
														T_{\varepsilon_K^r}	&	m\text{ even}\\
														0					 &	m\text{ odd}
													\end{cases}
\] 
is a $G_\QQ$-stable $\calO_E$-lattice in $V_m$. Let $A_m=V_m/T_m$. Since Selmer groups decompose under direct sums, we make the following definition for the mixed plus and minus Selmer groups of $A_m^\ast$.
\begin{definition}
	Let $\mf{s}\in\mf{S}$. Under Hypothesis \ref{hyp2}, we define
	\[ H^{1,\mf{s}}_f(k_\infty,A_m^\ast):=\left(\bigoplus_{i=0}^{\wt{r}-1}H^{1,\mf{s}_i}_f\left(k_\infty,A_{f_i}^\ast\left((r-i)(k-1)\right)\right)\right)\oplus \begin{cases}
														H^1_f(k_\infty,A_{\varepsilon_K^r}^\ast)	&	m\text{ even}\\
														0					 &	m\text{ odd}
													\end{cases}
	\]
\end{definition}

%++++++++++++++++++++++++++++++++++++++++++++++++++++++++++++++

\subsection{Main conjectures}
We begin by recalling the results of \cite{mazurwiles,lei09} on the main conjectures for Dirichlet characters and CM newforms.

Write
\[ L_{\triv,a}=\frac{F_a}{G_a}
\]
with $F_a,G_a\in\Lambda_{\Zp}(\Gamma)$ and
\[ G_a=\begin{cases}
				\gamma_0-1					&a=0,	\\
				\gamma_0-\langle\gamma_0\rangle	&a=1,	\\
				1							&\text{otherwise.}
			\end{cases}
\]
To streamline the statement of the main conjecture, let
\[ \wt{L}_\triv:=\sum_{a\in\ZZ/(p-1)\ZZ}\pi_{\omega^a}F_a
\]
and $\wt{L}_{\epsilon_K}:=L_{\epsilon_K}$. The results of Mazur and Wiles are equivalent to the following theorem.
\begin{theorem}[\cite{mazurwiles}]\label{thm:MW}
Let $\theta=\triv$ or $\epsilon_K$. Then, for all $a\in\ZZ/(p-1)\ZZ$,
\[
	\Char_{\Lambda_{\Zp}(\Gamma)}\left(\left(H^1_f(k_\infty,A_\theta^\ast)^\vee\right)^{\omega^a}\right)=\left((\Tw_1(\wt{L}_\theta)^{\omega^a}\right).
\]
\end{theorem}
\begin{proof}
	In view of \eqref{eqn:thetaSelmerFinfty}, the results of Mazur--Wiles translate into
	\[ \Char_{\Lambda_{\Zp}(\Gamma)}\left(\left(H^1_f(F_\infty,A_\theta^\ast)^\vee\right)^{\theta\omega^a}\right)=\left((\Tw_1(\wt{L}_\theta))^{\omega^a}\right).
	\]
	Our statement then follows from \eqref{eqn:thetaSelmer}.
\end{proof}

Let 
\[
\wt{L}^\pm_{f_i}:=\sum_{a\in\ZZ/(p-1)\ZZ} \frac{\pi_{\omega^a}L^\pm_{f_i}}{\Tw_{-1}\left(\fl_{f_i,a}^\pm\right)},
\]
 where $\fl_{f_i,a}^\pm$ describe the images of the plus and minus Coleman maps, as explained in the appendix. The following theorem was proved by the second author.
\begin{theorem}[{\cite[Section~6]{lei09}}]\label{thm:MCCM}
Assume Hypothesis~\ref{hyp2}. Then, for all $a\in\ZZ/(p-1)\ZZ$, the module $\left(H^{1,\pm}_f\left(k_\infty,A_{f_i}^\ast((r-i)(k-1))\right)^\vee\right)^{\omega^a}$ is $\Lambda_{\Qp}(\Gamma)$-torsion and
\[
	\Char_{\Lambda_{E}(\Gamma)}\left(\left(H^{1,\pm}_f\left(k_\infty,A_{f_i}^\ast((r-i)(k-1))\right)^\vee\right)^{\omega^a}\right)\supset\left((\Tw_{(r-i)(k-1)+1}(\wt{L}^{\pm}_{f_i}))^{\omega^a}\right).
\]
Moreover, Kato's main conjecture for $f_i$ holds if and only if equality holds.
\end{theorem}

For $\mf{s}\in\mf{S}$, define
\[
\wt{L}_{V_m}^\mf{s}=\left(\prod_{i=0}^{\wt{r}-1}\Tw_{(r-i)(k-1)}\left(\wt{L}_{f_i}^{\mf{s}_i}\right)\right)\cdot
	\begin{cases}
		\wt{L}_{\varepsilon_K^r}	& m\text{ even,} \\
		1				& m\text{ odd.}
	\end{cases}
\]
We are now ready to formulate the Main Conjecture of Iwasawa Theory for $V_m$.
\begin{conjecture}\label{MC}
For all $\mf{s}\in\mf{S}$ and $a\in\ZZ/(p-1)\ZZ$, we have equality
\[
	\Char_{\Lambda_{E}(\Gamma)}\left(\left(H^{1,\mf{s}}_f\left(k_\infty,A_m^\ast\right)^\vee\right)^{\omega^a}\right)=\left((\Tw_1(\wt{L}^{\mf{s}}_{V_m}))^{\omega^a}\right).
\]
\end{conjecture}

We have the following result towards this conjecture.

\begin{theorem}
	Assume Hypothesis~\ref{hyp2}. Then, for all $\mf{s}\in\mf{S}$ and $a\in\ZZ/(p-1)\ZZ$, $\left(H^{1,\mf{s}}_f\left(k_\infty,A_m^\ast\right)^\vee\right)^{\omega^a}$ is $\Lambda_{\Qp}(\Gamma)$-torsion and
\[
	\Char_{\Lambda_{E}(\Gamma)}\left(\left(H^{1,\mf{s}}_f\left(k_\infty,A_m^\ast\right)^\vee\right)^{\omega^a}\right)\supset\left((\Tw_1(\wt{L}^{\mf{s}}_{V_m}))^{\omega^a}\right).
\]
\end{theorem}
\begin{proof}
	Since
	\[ \Char_{\Lambda_{E}(\Gamma)}(M_1\oplus M_2)=\Char_{\Lambda_{E}(\Gamma)}(M_1)\oplus\Char_{\Lambda_{\Qp}(\Gamma)}(M_2),
	\]
	the theorem follows from Theorems \ref{thm:MW} and \ref{thm:MCCM}.
\end{proof}

\begin{remark}
It is clear that Conjecture~\ref{MC} holds if and only if the main conjecture for $f_i$ holds for each $i$. In \cite{pollackrubin04}, the main conjecture was proved for elliptic curves defined over $\QQ$ and $a=0$. As stated in {\it op.\ cit.}, the same proof in fact works for all $a$. In \cite[Section~7]{lei09}, it is shown that one could similarly prove the main conjecture for all CM forms with rational coefficients under the assumption that $K$ is of class number $1$. Therefore, the main conjecture for $V_m$ holds when $f$ is defined over $\QQ$ and $K$ has class number $1$.
\end{remark}

%++++++++++++++++++++++++++++++++++++++++++++++++++++++++++++++
%++++++++++++++++++++++++++++++++++++++++++++++++++++++++++++++

\section{The admissible \texorpdfstring{$p$-adic $L$-functions}{p-adic L-functions}}\label{sec:exceptional}

In this section, we construct $p$-adic $L$-functions for the symmetric powers of CM newforms for non-ordinary primes, thus verifying the conjecture of Dabrowski--Panchishkin (\cite[Conjecture 1]{dabrowski11},\cite[Conjecture 6.2]{panchishkin94}). In \cite{lei10}, the second author used elliptic units to construct two $p$-adic $L$-functions for the symmetric square lying in $\HH_{k-1,E}(G_\infty)$, but only proved that they interpolate the $L$-values $L(\Sym^2(V_f),\theta^{-1},2k-2)$ (or equivalently $L(V_2,\theta^{-1},k-1)$). Here, we construct $2^{\wt{r}}$ $p$-adic $L$-functions for $V_m$ and we show that they have the expected interpolation property at all critical twists and the expected growth rate. We also relate them to the mixed plus and minus $p$-adic $L$-functions $L_{V_m}^\mf{s}$ introduced in Definition \ref{def:mixed_plus_minus}, thus providing an interesting analogue of Pollack's decomposition of $p$-adic $L$-functions of non-ordinary modular forms. As in \cite{harron12}, these $p$-adic $L$-functions are defined as products of $p$-adic $L$-functions of newforms and Dirichlet characters as given by the decomposition in Corollary \ref{cor:decompoL}. This approach was also taken by Dabrowski in \cite{dabrowski93} for $\Sym^m(V_f)$. We then go on to prove the exceptional zero conjecture for $V_m$ using a recent result of Denis Benois (\cite{benois??}).

%++++++++++++++++++++++++++++++++++++++++++++++++++++++++++++++

\subsection{Definition and basic properties}
Recall that each newform $f_i$ in the decomposition of Corollary~\ref{cor:decompoL} has weight $k_i:=(m-2i)(k-1)+1$ and let $\epsilon_i$ denote its Nebentypus. It follows from the proof of Lemma \ref{lem:DcrisVm} that the roots of $x^2+\epsilon_i(p)p^{(m-2i)(k-1)}$ are $\pm p^{(r-i)(k-1)}$ if $m$ is even and $\pm\alpha p^{(r-i)(k-1)}$ if $m$ is odd. Accordingly, we let
\begin{equation}\label{eqn:alphai}
	\alpha_{i,\pm}:=\begin{cases}
				\pm p^{(r-i)(k-1)}	& \text{if }m\text{ is even,}\\
				\pm\alpha p^{(r-i)(k-1)}& \text{if }m\text{ is odd.}
			\end{cases}
\end{equation}
Since $h_i:=\ord_p(\alpha_{i,\pm})<k_i-1$, the result of Amice--V\'elu (Theorem \ref{thm:AV}) provides for each choice of $i$ and $\pm$ a $p$-adic $L$-function $L_{f_i,\pm}\in\mathcal{H}_{h_i,E}(G_\infty)$ (which should not be confused with the notation for Pollack's $p$-adic $L$-functions $L_{f_i}^\pm$). Recall the notation concerning elements of $\mf{S}$ as introduced in section \S\ref{sec:mixed_plus_minus}. Corollary~\ref{cor:decompoL} suggests the following definition.

\begin{definition}\label{def:admissible_p-adic_L}
For each $\mf{s}\in\mf{S}$, define
\begin{equation}
	L_{V_m,\mf{s}}=\left(\prod_{i=0}^{\wt{r}-1}\Tw_{(r-i)(k-1)}\left(L_{f_i,\mf{s}_i}\right)\right)\cdot\begin{cases}
					\displaystyle L_{\epsilon_K^r}	&\text{if }m\text{ is even,}\\
					1	&\text{if }m\text{ is odd.}
				\end{cases}
\end{equation}
\end{definition}
This gives $2^{\wt{r}}$ $p$-adic $L$-functions for $V_m$.
\begin{theorem}\label{thm:p-adicL_Vm}
	For each $\mf{s}\in\mf{S}$ and each $(\theta,j)\in C_m$, where $\theta$ has conductor $p^n$, we have
	\begin{equation}\label{eq:formula}
		\theta\chi^j(L_{V_m,\mf{s}})=e_{m,\mf{s}}(\theta,j)\frac{L(V_m,\theta^{-1},j)}{\Omega_m(\theta,j)},
	\end{equation}
	where
	\begin{align*}
		&e_{m,\mf{s}}(\theta,j)=\left(\prod_{i=0}^{\wt{r}-1}e_{f_i,\alpha_{i,\mf{s}_i}}(\theta,j+(r-i)(k-1))\right)\cdot
			\begin{cases}
				e_{\epsilon_K^r}(\theta,j)	&m\text{ even,} \\
				1					&m\text{ odd,}
			\end{cases}
			\\
			&=\begin{cases}
				\begin{array}{ll}
					\displaystyle\frac{1}{p^{n(k-1)\frac{r(r+1)}{2}}} &\hspace{-1em}\left((-1)^n\!\left(1-\theta^{-1}(p)p^{-j}\right)\left(1+\theta(p)p^{j-1}\right)\right)^{\mf{s}^-} \\
							&\times\left(\left(1+\theta^{-1}(p)p^{-j}\right)\left(1-\theta(p)p^{j-1}\right)\right)^{\mf{s}^+}\!\! e_{\epsilon_K^r}(\theta,j)
				\end{array}
				& m\text{ even,}
				\\
				\begin{array}{ll}
					\displaystyle\frac{1}{\left(\alpha^{r+1}p^{(k-1)\frac{r(r+1)}{2}}\right)^n} &\hspace{-1em}\left((-1)^n\!\cdot\left(1-\theta^{-1}(p)\alpha p^{-j}\right)\left(1+\frac{\theta(p)p^{j-1}}{\alpha}\right)\right)^{\mf{s}^-} \\
						&\times\left(\left(1+\theta^{-1}(p)\alpha p^{-j}\right)\left(1-\frac{\theta(p)p^{j-1}}{\alpha}\right)\right)^{\mf{s}^+}
				\end{array}
				& m\text{ odd}
			\end{cases}
	\end{align*}
	and $\Omega_m(\theta,j)$ is as in Definition \ref{def:deltacritical}.

	Furthermore,
	\begin{equation} L_{V_m,\mf{s}}\in\mc{H}_{(k-1)\frac{d^+d^-}{2},E}(G_\infty)\label{eq:growth}
	\end{equation}
	unless $4|m$, in which case
	\[ (\gamma_0-1)(\gamma_0-\chi(\gamma_0))L_{V_m,\mf{s}}\in\mc{H}_{(k-1)\frac{d^+d^-}{2},E}(G_\infty).
	\]
\end{theorem}
\begin{remark}
	When $m\geq2$, $L_{V_m,\mf{s}}$ is \textit{not} uniquely determined by its interpolation property. Indeed, an element of $\mc{H}_a$ needs to satisfy an interpolation property at $\theta\chi^j$ for at least $\lfloor a+1\rfloor$ choices of $j$. When $m\geq4$ is even, $(k-1)\frac{d^+d^-}{2}\geq3(k-1)$ is greater than the number of distinct $j$ such that $(\theta,j)\in C_m$ (this number being $2(k-1)$). Similarly, when $m$ is odd, the latter number is only $k-1$ which is less than $(k-1)\frac{d^+d^-}{2}\geq2(k-1)$. When $m=2$, $(k-1)\frac{d^+d^-}{2}=2(k-1)$, but the parity condition on $(\theta,j)$ implies that only half of the $j$'s can be used.
\end{remark}
\begin{proof}
	One can easily verify that $C_m$ is a subset of all the pairs $(\theta,j)$ at which the elements $\Tw_{(r-i)(k-1)}\left(L_{f_i,\mf{s}_i}\right)$ and $L_{\epsilon_K^r}$ satisfy the interpolation properties of Theorems \ref{thm:AV} and \ref{thm:KL}, respectively. The interpolation property for $L_{V_m,\mf{s}}$ then follows immediately from its definition.

	As for the growth condition, Theorems \ref{thm:AV} and \ref{thm:KL} tell us that
	\[ \Tw_{(r-i)(k-1)}\left(L_{f_i,\mf{s}_i}\right)\in\mc{H}_{(\frac{m}{2}-i)(k-1),E}(G_\infty)
	\]
	(since the twisting operation does not affect the growth) and $L_{\epsilon_K^r}\in\Frac(\mc{H}_0)$. When taking a product, the growth rates are additive; hence, the growth of $L_{V_m,\mf{s}}$ is
	\[ \sum_{i=0}^{\wt{r}-1}\left(\frac{m}{2}-i\right)(k-1).
	\]
	A simple summation then gives the stated growth rate. The statements about being in $\mc{H}_{(k-1)\frac{d^+d^-}{2},E}(G_\infty)$ rather than its fraction field follow by the corresponding statements in Theorems \ref{thm:AV} and \ref{thm:KL}.
\end{proof}

\begin{proposition}
Let $\omega^a$ be a character of $\Delta$. For each $\mf{s}$, the $\omega^a$-isotypical component of $L_{V_m,\mf{s}}$ is non-zero. 
\end{proposition}
\begin{proof}
This follows from the same proof as Proposition~\ref{prop:nonzero}.
\end{proof}

%++++++++++++++++++++++++++++++++++++++++++++++++++++++++++++++

\subsection{A conjecture of Dabrowski--Panchishkin}
	We begin this section by computing the generalized Hasse invariant $h_p(V_m)$ of $V_m$. Recall that $h_p(V_m):=\max(P_N(d^\pm,V_m)-P_H(d^\pm,V_m))$ (see e.g.\ \cite[\S5]{panchishkin94}). We obtain the following from lemmas \ref{lem:HodgePoly} and \ref{lem:NewtonPoly}.

	\begin{lemma}
		For all $m\geq2$,
		\[ h_p(V_m)=(k-1)\frac{d^+d^-}{2}.
		\]
	\end{lemma}
	\begin{proof}
		If $m$ is even, $P_N(x,V_m)$ is identically zero, so $h_p(V_m)=\max(-P_H(d^\pm,V_M))$. A simple computation shows that
		\[ h_p(V_m)=-P_H(r+1,V_m)=(k-1)\frac{r(r+1)}{2},
		\]
		as desired. When $m$ is odd, we have $d^+=d^-=r+1$. Thus,
		\[ h_p(V_m)=P_N(r+1,V_m)-P_H(r+1,V_m)=(r+1)\frac{k-1}{2}-(k-1)\frac{(r+1)(-r)}{2},
		\]
		which yields the result.
	\end{proof}
	This lemma, together with \eqref{eq:growth} of Theorem \ref{thm:p-adicL_Vm}, shows that $L_{V_m,\mf{s}}$ has the growth property predicted by \cite[Conjecture 1(iv)]{dabrowski11} and \cite[Conjecture 6.2(iv)]{panchishkin94}, and \eqref{eq:formula} of Theorem \ref{thm:p-adicL_Vm} shows that it satisfies the expected interpolation property.

%++++++++++++++++++++++++++++++++++++++++++++++++++++++++++++++

\subsection{Decomposition into mixed plus and minus \texorpdfstring{$p$-adic $L$-functions}{p-adic L-functions}}

We generalize the decomposition of Theorem \ref{thm:pollack}, due to Pollack, to our setting. In other words, we decompose each of the $2^{\wt{r}}$ $p$-adic $L$-functions $L_{V_m,\mf{s}}$ introduced in Definition \ref{def:admissible_p-adic_L} as a linear combination of the $2^{\wt{r}}$ $p$-adic $L$-functions $L_{V_m}^\mf{s}$ given in Definition \ref{def:mixed_plus_minus}.

Let
\[
	\ell_i^+:=\log^+_{k_i-1}\text{ and }\ell_i^-:=\alpha_{i,+}\log^-_{k_i-1}.
\]
Then, \eqref{eqn:Lf+-} becomes,
\begin{equation}
	L_{f_i}^{\mf{s}_i}=\frac{L_{f_i,+}+\mf{s}_iL_{f_i,-}}{2\ell_i^{\mf{s_i}}}\label{eq:L+-},
\end{equation}
and we have the following result.
\begin{lemma}\label{lem:expansion}
For all $\mf{s}\in\mf{S}$, we have 
\begin{equation}\label{eq:expand}
2^{\wt{r}}\ell_{V_m}^\mf{s}L_{V_m}^\mf{s}=\sum_{\mf{t}\in\mf{S}}a_{\mf{s},\mf{t}}L_{V_m,\mf{t}}
\end{equation}
where
\[ \ell_{V_m}^\mf{s}=\prod_{i=0}^{\wt{r}-1}\Tw_{(r-i)(k-1)}\left(\ell_i^{\mf{s}_i}\right)
\]
and $a_{\mf{s},\mf{t}}\in\{+1,-1\}$. Moreover, $a_{\mf{s},\mf{t}}$ is given by $(-1)^{b_{\mf{s},\mf{t}}}$ where $b_{\mf{s},\mf{t}}$ is the number of $i\in[0,\wt{r}-1]$ such that $\mf{s}_i=\mf{t}_i=-$.
\end{lemma}
\begin{proof}
On substituting \eqref{eq:L+-} into the definition of $L_{V_m}^{\mf{s}}$, we have
\[
L_{V_m}^\mf{s}=\left(\prod_{i=0}^{\wt{r}-1}\Tw_{(r-i)(k-1)}\left(\frac{L_{f_i,+}+\mf{s}_iL_{f_i,-}}{2\ell_i^{\mf{s}_i}}\right)\right)\cdot\begin{cases}
					\displaystyle L_{\epsilon_K^r}	&\text{if }m\text{ is even,}\\
					1	&\text{if }m\text{ is odd.}
				\end{cases}
\]
Hence, upon expanding, we obtain \eqref{eq:expand}. To find $a_{\mf{s},\mf{t}}$, we observe that the sign in front of $L_{V_m,\mf{t}}$ is given by
\[
\prod_{\substack{i=0 \\ \mf{t}_i=-}}^{\wt{r}-1}\mf{s}_i,
\]
so we are done.
\end{proof}

Let $\mf{A}$ be the $2^{\wt{r}}\times 2^{\wt{r}}$ matrix $(a_{\mf{s},\mf{t}})_{(\mf{s},\mf{t})\in\mf{S}\times\mf{S}}$.

\begin{proposition}
The matrix $\mf{A}$ is symmetric and has orthogonal columns and rows.
\end{proposition}
\begin{proof}
It is clear from the description of $a_{\mf{s},\mf{t}}$ in Lemma~\ref{lem:expansion} that $\mf{A}$ is symmetric. For the orthogonality, we show that for $\mf{t}\ne\mf{u}$,
\begin{equation}\label{eq:half}
\#\{\mf{s}\in\mf{S}:a_{\mf{s},\mf{t}}=a_{\mf{s},\mf{u}}\}=\#\{\mf{s}\in\mf{S}:a_{\mf{s},\mf{t}}=-a_{\mf{s},\mf{u}}\}=2^{\wt{r}-1},
\end{equation}
which would imply that the $(\mf{t},\mf{u})$-entry
\[
\sum_{\mf{s}\in\mf{S}}a_{\mf{s},\mf{t}}a_{\mf{s},\mf{u}}
\]
of $\mf{A}\mf{A}^\text{T}$ is zero

Let $I=\{i:\mf{t}_i\ne\mf{u}_i\}$. By definition, $\#I\ge1$. Let $\mf{s}\in\mf{S}$ be such that $a_{\mf{s},\mf{t}}=a_{\mf{s},\mf{u}}$. By Lemma~\ref{lem:expansion},
\begin{equation}\label{eq:parity}
b_{\mf{s},\mf{t}}\equiv b_{\mf{s},\mf{u}}\mod2.
\end{equation}
For each $i\in I$, either $(\mf{t}_i,\mf{u}_i)=(+,-)$ or $(-,+)$. Hence, each $i\in I$ such that $\mf{s}_i=-$ contributes a ``one" to one side of the equation in \eqref{eq:parity}, but not to the other side. But for $i\notin I$, each $\mf{s}_i=-$ would contribute a ``one" to both sides.  Therefore, in order for \eqref{eq:parity} to hold, $\#\{i\in I:\mf{s}_i=-\}$ must be even, but for $i\notin I$, no such condition is required. Hence, the number of choices coming from $i\in I$ is given by
\[ \sum_{n\text{ even}}\binom{\#I}{n}=2^{\#I-1}
\]
and for $i\not\in I$, there are $2^{\wt{r}-\#I}$ choices. This gives $2^{\wt{r}-1}$ choices for $\mf{s}$ such that $a_{\mf{s},\mf{t}}=a_{\mf{s},\mf{u}}$. Since this is half of $\mf{S}$, the other half must have $a_{\mf{s},\mf{t}}=-a_{\mf{s},\mf{u}}$.
\end{proof}

\begin{corollary}\label{cor:inverse}
The matrix $\mf{A}$ is invertible with inverse $2^{-\wt{r}}\mf{A}$.
\end{corollary}
\begin{proof}
Since all entries of $\mf{A}$ are $+1$ or $-1$, the fact that $\mf{A}$ is symmetric and has orthogonal columns and rows implies that $\mf{A}^2$ is a diagonal matrix. The $(\mf{t},\mf{t})$-entry of $\mf{A}^2$ is
\[ \sum_{\mf{s}\in\mf{S}}a_{\mf{s},\mf{t}}^2=2^{\wt{r}},
\]
hence the result.
\end{proof}

\begin{corollary}
For all $\mf{t}\in\mf{S}$, we have an expansion
\[
L_{V_m,\mf{t}}=\sum_{s\in\mf{S}}a_{\mf{s},\mf{t}}\ell_{V_m}^\mf{s}L_{V_m}^\mf{s}.
\]
\end{corollary}
\begin{proof}
This follows immediately from Lemma~\ref{lem:expansion} and Corollary~\ref{cor:inverse}.
\end{proof}

%++++++++++++++++++++++++++++++++++++++++++++++++++++++++++++++

\subsection{Exceptional zeros and \texorpdfstring{$L$}{L}-invariants}\label{subsec:exceptional}
	In this section, we will use the main result of \cite{benois??} to determine the (analytic) $L$-invariants of $V_m$. It follows that these $L$-invariants are given by Benois's (arithmetic) $L$-invariant (as defined in \cite{benois11}). We also discuss a case that does not fit into the standard situation of trivial zeroes, namely when $4|m$ we can say something at a couple of non-critical twists.

	From now on, if $m$ is odd and $\alpha=\pm p^{\frac{k-1}{2}}$, then we fix the choice $\alpha=p^{\frac{k-1}{2}}$. We can read off the following from Theorems \ref{thm:KL} and \ref{thm:AV}.

	\begin{lemma}\label{lem:TrivZeroesOfPieces}\mbox{}
		\begin{enumerate}
			\item The $p$-adic $L$-function $L_{\epsilon_K^r}$ has no trivial zeroes.
			\item If $m$ is even, then $\Tw_{(r-i)(k-1)}\left(L_{f_i,+}\right)$ (resp., $\Tw_{(r-i)(k-1)}L_{f_i,-}$) has a trivial zero at the critical twist $(\theta,j)$ if and only if $j=1$ (resp., $j=0$) and $\theta=\mathbf{1}$.
			\item If $m$ is odd, then $\Tw_{(r-i)(k-1)}\left(L_{f_i,+}\right)$ (resp., $\Tw_{(r-i)(k-1)}L_{f_i,-}$) has a trivial zero at the critical twist $(\theta,j)$ if and only if $j=\frac{k+1}{2}$ (resp., $j=\frac{k-1}{2}$), $\theta=\mathbf{1}$, $k$ is odd, and $\alpha=p^{\frac{k-1}{2}}$.
		\end{enumerate}
	\end{lemma}
	\begin{proof}
		In all cases, $\theta$ must be trivial, otherwise $\theta(p)=0$. For $L_{\epsilon_K^r}$, a trivial zero must occur at $j=0$ or $1$ and requires that $\epsilon_K^r(p)=1$. Since $p$ is inert in $K$, we have $\epsilon_K(p)=-1$, so $r$ needs to be even. But for even $r$, $(\mathbf{1},j)\not\in C_m$ for $j=0,1$.

		For the $f_i$, in all cases,
		\begin{equation}\label{eqn:efi}
			e_{f_i,\alpha_{i,\pm}}(\mathbf{1},j)=\left(1-\alpha_{i,\mp}p^{-j}\right)\left(1-\frac{p^{j-1}}{\alpha_{i,\pm}}\right).
		\end{equation}
		For $m$ even, \eqref{eqn:alphai} tells us that this is equal to
		\[ \left(1\pm p^{\frac{k_i-1}{2}-j}\right)\left(1\mp p^{j-\frac{k_i+1}{2}}\right)
		\]
		since $(r-i)(k-1)=\frac{k_i-1}{2}$. Hence, $L_{f_i,-}$ has a trivial zero at $j=\frac{k_i-1}{2}$ and $L_{f_i,+}$ has one at $j=\frac{k_i+1}{2}$. The twist operation simply shifts these to the left by $\frac{k_i-1}{2}$. When $m$ is odd, the proof is similar. In this case, we can write $\alpha=\zeta p^{\frac{k-1}{2}}$ where $\zeta$ is a root of unity. Then, using \eqref{eqn:alphai} shows that \eqref{eqn:efi} becomes
		\[ \left(1\pm\zeta p^{\frac{k_i-1}{2}-j}\right)\left(1\mp\zeta p^{j-\frac{k_i+1}{2}}\right)
		\]
		since $(r-i)(k-1)+\frac{k-1}{2}=\frac{k_i-1}{2}$. If $k$ is even, then $k_i$ is odd, so $\frac{k_i\pm1}{2}\not\in\ZZ$, so no trivial zeroes can occur. Assume that $k$ is odd. Then, $\zeta$ must be $\pm1$, in which case we have fixed the choice $\zeta=1$. In this case, $L_{f_i,-}$ has a trivial zero at $j=\frac{k_i-1}{2}$ and $L_{f_i,+}$ has one at $j=\frac{k_i+1}{2}$. The twist operation shifts these to the left by $\frac{k_i-1}{2}-\frac{k-1}{2}$.
	\end{proof}

	For the character $\omega^a$ of $\Delta$ and an element $L\in\Frac(\HH_\infty(G_\infty))$, we write
	\[ L_a(s):=\langle \pi_{\omega^a}L\rangle^s
	\]
	for the $a$th branch of $L$. This is a meromorphic function of $s\in\Zp$ and for $j\in\ZZ$, $L_a(j)=\omega^{a-j}\chi^j(L)$.

	The following is a special case of the main theorem of \cite{benois??}.
	
	\begin{theorem}[Theorem 4.3.2 of \cite{benois??}]\label{thm:benois}
		Let $f^\prime$ be a newform of prime-to-$p$ level, Nebentypus $\epsilon^\prime$, and odd weight $k^\prime$ with $a_p(f^\prime)=0$. Suppose $\alpha^\prime=p^{\frac{k^\prime-1}{2}}$ is a root of $x^2+\epsilon^\prime(p)p^{k^\prime-1}$. Then, for $a=\frac{k^\prime+1}{2}$ or $\frac{k^\prime-1}{2}$, there exists $\mc{L}_{V_{f^\prime},\alpha^\prime}(\mathbf{1},a)\in\CC_p$ such that, 
		\[ \lim_{s\rightarrow a}\frac{L_{f^\prime,\alpha^\prime,a}(s)}{s-a}=\mc{L}_{V_{f^\prime},\alpha^\prime}(\mathbf{1},a)\left(1+\frac{1}{p}\right)\frac{L(f^\prime,a)}{\Omega_{f^\prime}(\mathbf{1},a)}.
		\]
		Furthermore, $\mc{L}_{V_{f^\prime},\alpha^\prime}(\mathbf{1},a)$ is given by Benois' arithmetic of $L$-invariant $(V_f^\prime,\alpha^\prime)$ at $(\mathbf{1},a)$ as defined in \cite{benois11}.
	\end{theorem}
	\begin{proof}
		Here are some remarks that explain how the above statement follows from \cite[Theorem 4.3.2]{benois??}. First, what we denote $\mc{L}_{V_{f^\prime},\alpha^\prime}(\mathbf{1},\frac{k^\prime+1}{2})$ is denoted $-\mathscr{L}_{\alpha^\prime}(f^\prime)$ by Benois. Since we assume that $a_p(f^\prime)=0$, we know that the other root of $x^2+\epsilon^\prime(p)p^{k-1}$ is $\ol{\alpha}^\prime=-\alpha^\prime$, hence $\epsilon^\prime(p)=-1$ and $\vp$ acts semisimply on $\Dcris(V_f)$.  What we write $L_{f^\prime,\alpha^\prime,a}(s)$ is denoted $L_{p,\alpha^\prime}(f,\omega^a,s)$ in \cite{benois??}. Finally, to obtain the statement for $a=\frac{k^\prime-1}{2}$ from \cite[Theorem 4.3.2]{benois??}, one uses the compatibility of Benois' $L$-invariant with the $p$-adic functional equation as discussed in \cite[\S0.3, Remark 3]{benois??}.
	\end{proof}
	
	Note that for $a\in\ZZ/(p-1)\ZZ$, the $p$-adic $L$-function $L_{V_m,\mf{s},a}(s)$ is an analytic function of $s\in\Zp$ unless $4|m$, in which case it follows from Theorem \ref{thm:p-adicL_Vm} that simple poles arise for $L_{V_m,\mf{s},1}(s)$ at $s=1$ and for $L_{V_m,\mf{s},0}(s)$ at $s=0$. Translating $C_m$ into this setting, we obtain the following lemma.
	
	\begin{lemma}
		An integer $j$ is critical for $L_{V_m,\mf{s},a}$ if and only if
		\[ \begin{cases}
				-(k-1)+1\leq j\leq 0	&\text{if }m\text{ is even and }a\text{ and }r\text{ have opposite parities,} \\
				1\leq j\leq k-1		&\text{otherwise.}
			\end{cases}
		\]
	\end{lemma}
	\begin{proof}
		Evaluating $L_{V_m,\mf{s},a}(s)$ at $s=j$ is the same as evaluating $\omega^{a-j}\chi^j(L_{V_m,\mf{s}})$.
	\end{proof}

	If $m$ is even, then for $a=0$ or $1$, let $\sigma_a=\sgn(a-\frac{1}{2})$ and 
	\begin{equation}
		\mc{L}_{V_m,\mf{s}}(\mathbf{1},a):=\prod_{\substack{i=0 \\ \mf{s}_i=\sigma_a}}^{\wt{r}-1}\mc{L}_{V_{f_i,\alpha_{i,\mf{s}_i}}}\left(\mathbf{1},a+\frac{k_i-1}{2}\right).
	\end{equation}
	If $m$ is odd, with $k$ odd and $\alpha=p^{\frac{k-1}{2}}$, then for $a=0$ or $1$, let 
	\begin{equation}
		\mc{L}_{V_m,\mf{s}}\left(\mathbf{1},a+\frac{k-1}{2}\right):=\prod_{\substack{i=0 \\ \mf{s}_i=\sigma_a}}^{\wt{r}-1}\mc{L}_{V_{f_i,\alpha_{i,\mf{s}_i}}}\left(\mathbf{1},a+\frac{k_i-1}{2}\right).
	\end{equation}

	\begin{theorem}
		Suppose $j$ is critical for $L_{V_m,\mf{s},a}(s)$.
		\begin{enumerate}
			\item If $m$ is even, then
		\[
			\ord_{s=j}L_{V_m,\mf{s},a}(s)\text{ is }\begin{cases}
												\geq\mf{s}^+	&\text{if }a=j=1\text{ (and }r\text{ is odd)}, \\
												\geq\mf{s}^-	&\text{if }a=j=0\text{ (and }r\text{ is odd)}, \\
												=0			&\text{otherwise.}
											\end{cases}
		\]
		Furthermore, for $r$ odd,
		\[
			\lim_{s\rightarrow1}\frac{L_{V_m,\mf{s},1}(s)}{(s-1)^{\mf{s}^+}}=2^{\mf{s}^-+1}\!\cdot\left(1-\frac{1}{p}\right)^{\mf{s}^-}\!\!\!\cdot\left(1+\frac{1}{p}\right)^{\mf{s}^+}\!\!\!\cdot\mc{L}_{V_m,\mf{s}}(\mathbf{1},1)\cdot\frac{L(V_m,1)}{\Omega_m(\mathbf{1},1)}
		\]
		and
		\[
			\lim_{s\rightarrow0}\frac{L_{V_m,\mf{s},0}(s)}{s^{\mf{s}^-}}=2^{\mf{s}^++1}\!\cdot\left(1-\frac{1}{p}\right)^{\mf{s}^+}\!\!\!\cdot\left(1+\frac{1}{p}\right)^{\mf{s}^-}\!\!\!\cdot\mc{L}_{V_m,\mf{s}}(\mathbf{1},0)\cdot\frac{L(V_m,0)}{\Omega_m(\mathbf{1},0)}.
		\]
			\item If $m$ is odd, then
		\[
			\ord_{s=j}L_{V_m,\mf{s},a}(s)\text{ is }\begin{cases}
												\geq\mf{s}^+	&\displaystyle\text{if }k\text{ is odd, }\alpha=p^{\frac{k-1}{2}},\text{ and }a=j=\frac{k+1}{2}, \\
												\geq\mf{s}^-	&\displaystyle\text{if }k\text{ is odd, }\alpha=p^{\frac{k-1}{2}},\text{ and }a=j=\frac{k-1}{2}, \\
												=\,\,?			&\displaystyle\text{if }k\text{ is even, }j=\frac{k}{2},\text{ and }L\left(V_m,\omega^{\frac{k}{2}-a},\frac{k}{2}\right)=0 \\
												=0			&\text{otherwise.}
											\end{cases}
		\]
		Furthermore, in the first two cases,
		\[
			\lim_{s\rightarrow\frac{k+1}{2}}\frac{L_{V_m,\mf{s},\frac{k+1}{2}}(s)}{\left(s-\frac{k+1}{2}\right)^{\mf{s}^+}}=2^{\mf{s}^-}\!\cdot\left(1-\frac{1}{p}\right)^{\mf{s}^-}\!\!\!\cdot\left(1+\frac{1}{p}\right)^{\mf{s}^+}\!\!\!\cdot\mc{L}_{V_m,\mf{s}}\left(\mathbf{1},\frac{k+1}{2}\right)\cdot\frac{L\left(V_m,\frac{k+1}{2}\right)}{\Omega_m\left(\mathbf{1},\frac{k+1}{2}\right)}
		\]
		and
		\[
			\lim_{s\rightarrow\frac{k-1}{2}}\frac{L_{V_m,\mf{s},\frac{k-1}{2}}(s)}{\left(s-\frac{k-1}{2}\right)^{\mf{s}^-}}=2^{\mf{s}^+}\!\cdot\left(1-\frac{1}{p}\right)^{\mf{s}^+}\!\!\!\cdot\left(1+\frac{1}{p}\right)^{\mf{s}^-}\!\!\!\cdot\mc{L}_{V_m,\mf{s}}\left(\mathbf{1},\frac{k-1}{2}\right)\cdot\frac{L\left(V_m,\frac{k-1}{2}\right)}{\Omega_m\left(\mathbf{1},\frac{k-1}{2}\right)}.
		\]
		In the third case, we can say that the $p$-adic interpolation factor does not vanish.
		\end{enumerate}
	In all cases, $\mc{L}_{V_m,\mf{s}}(\mathbf{1},a)$ is given by Benois' arithmetic $L$-invariant as defined in \cite{benois11}.
	\end{theorem}
	\begin{remark}
		The inequalities in the theorem would become equalities if the corresponding $L$-invariants of the newforms $f_i$ were known to be non-zero. When $m$ is odd, $k$ is even, and $j=\frac{k}{2}$, we are at a central point which complicates matters since the order of vanishing of the archimedean $L$-function is a subtle point. Since the $p$-adic interpolation factor does not vanish, it would be natural to conjecture that
		\[ \ord_{s=\frac{k}{2}}L_{V_m,\mf{s},a}\left(s\right)=\ord_{s=\frac{k}{2}}L\left(V_m,\omega^{\frac{k}{2}-a},s\right).
		\]
	\end{remark}
	\begin{proof}
		Let $L_{i,\mf{s}_i,a}(s)=\Tw_{(r-i)(k-1)}L_{f_i,\mf{s}_i,a}(s)$. Suppose $m$ is even. The central point of the $L$-function of $V_m$ is at $s=1/2$, so none of the values we are considering are central. Therefore, none of the archimedean $L$-values we deal with vanish and all zeroes must come from the $p$-adic interpolation factor. Indeed, away from the near-central points, the non-vanishing of the archimedean $L$-functions at critical integers follows from their definition as an Euler product to the right of the critical strip $0<s<1$ and by the functional equation to the left of the critical strip (since we are only considering critical integers). At the near-central points, the non-vanishing is a classical theorem for $L(\epsilon_K,s)$ and a result of Jacquet--Shalika for newforms (\cite{jacquetshalika76}). We have that
		\[ L_{V_m,\mf{s},a}(s)=L_{\epsilon_K^r,a}(s)\cdot\prod_{i=0}^{\wt{r}-1}L_{i,\mf{s}_i,a}(s).
		\]
		We know from Lemma \ref{lem:TrivZeroesOfPieces} that $L_{\epsilon_K^r,a}(s)$ has no trivial zeroes at critical $j$ and similarly for $L_{i,\mf{s}_i,a}(s)$ when $(\mf{s}_i,a,j)\neq(+,1,1)$ or $(-,0,0)$. This shows that the order of vanishing of $L_{V_m,\mf{s},a}(j)$ is zero away from these two exceptional cases. Note that the exceptional cases do not correspond to critical $j$ when $r$ is even, so we may assume $r$ is odd from now on. For $a=0,1$, we get from Benois' result (Theorem \ref{thm:benois}) that near $s=a$,
		\[ L_{i,\sigma_a,a}(s)=(s-a)\mc{L}_{V_{f_i},\alpha_{i,\sigma_a}}\!\left(\mathbf{1},a+\frac{k_i-1}{2}\right)\left(1+\frac{1}{p}\right)\frac{L(f_i,a)}{\Omega_{f_i}(\mathbf{1},a)}+\text{higher order terms.}
		\]
		Since
		\[ L_{\epsilon_K,a}(a)=2\frac{L(\epsilon_K,a)}{\Omega_{\epsilon_K}(\mathbf{1},a)}
		\]
		and
		\[ L_{i,-\sigma_a,a}(a)=2\left(1-\frac{1}{p}\right)\frac{L\left(f_i,a+\frac{k_i-1}{2}\right)}{\Omega_{f_i}(\mathbf{1},a)},
		\]
		the statements in part (i) hold.

		Now, suppose $m$ is odd. If $k$ is even, we know from Lemma \ref{lem:TrivZeroesOfPieces} that no trivial zeroes occur. But $j=\frac{k}{2}\in\ZZ$ is the central point and the order of vanishing of the archimedean $L$-function at $j$ can be positive. If it is, we have no knowledge of the order of vanishing of the $p$-adic $L$-function, but if it isn't, we know the latter is non-vanishing at the central point. Outside of this case, we are at non-central points so the archimedean $L$-values are non-zero as explained above. We may therefore assume $k$ is odd. Using Lemma \ref{lem:TrivZeroesOfPieces} again, we know that the order of vanishing of $L_{V_m,\mf{s},a}$ is zero away from the two exceptional cases $(\alpha,\mf{s}_i,a,j)=\left(p^{\frac{k-1}{2}},+,\frac{k+1}{2},\frac{k+1}{2}\right)$ and $\left(p^{\frac{k-1}{2}},-,\frac{k-1}{2},\frac{k-1}{2}\right)$. The rest of the proof proceeds as in the case of even $m$.

		We know from \cite{benois11} that the $\mc{L}_{V_{f_i},\alpha_{i,\pm}}(\mathbf{1},a)$ are equal to Benois' arithmetic $L$-invariants. That the $\mc{L}_{V_m,\mf{s}}(\mathbf{1},a)$ are as well follows from the fact that Benois' $L$-invariant is multiplicative on direct sums of representations of $G_\QQ$.
	\end{proof}

	An interesting phenomenon occurs for $4|m$. In this case, each of the twisted $p$-adic $L$-functions of the $f_i$ has a trivial zero at $(\mathbf{1},1)$ or $(\mathbf{1},0)$, but $V_m$ itself is not critical at these characters because of the presence of the $p$-adic $L$-function of $\epsilon_K^r=\mathbf{1}$. However, this latter $p$-adic $L$-function is given by $\zeta_p(s)$ (resp. $\zeta_p(1-s)$) at the corresponding branches. Here $\zeta_p(s)$ is the $p$-adic Riemann zeta function of Kubota--Leopoldt and we know its residue at $s=1$.

	\begin{theorem}
		If $4|m$ and $j=0,1$, then
		\[
			\ord_{s=j}L_{V_m,\mf{s},a}(s)\text{ is }\begin{cases}
												\geq\mf{s}^+-1	&\text{if }a=j=1, \\
												\geq\mf{s}^--1	&\text{if }a=j=0, \\
												=0			&\text{otherwise.}
											\end{cases}
		\]
		Furthermore,
		\[
			\lim_{s\rightarrow1}\frac{L_{V_m,\mf{s},1}(s)}{(s-1)^{\mf{s}^+-1}}=2^{\mf{s}^-}\cdot\left(1-\frac{1}{p}\right)^{\mf{s}^-+1}\cdot\mc{L}_{V_m,\mf{s}}(\mathbf{1},1)\cdot\frac{L(V_m,1)}{\Omega_m(\mathbf{1},1)}
		\]
		and
		\[
			\lim_{s\rightarrow0}\frac{L_{V_m,\mf{s},0}(s)}{s^{\mf{s}^--1}}=-2^{\mf{s}^+}\cdot\left(1-\frac{1}{p}\right)^{\mf{s}^++1}\cdot\mc{L}_{V_m,\mf{s}}(\mathbf{1},0)\cdot\frac{L(V_m,0)}{\Omega_m(\mathbf{1},0)}.
		\]
		The inequalities would become equalities if the corresponding $L$-invariants of the $f_i$ were known to be non-zero.
	\end{theorem}
	\begin{proof}
		The proof is along the same lines as the previous theorem noting that near $s=1$
		\[ \zeta_p(s)=\left(1-\frac{1}{p}\right)\cdot\frac{1}{s-1}+\text{ higher order terms}.
		\]
	\end{proof}

\appendix
\section{Plus and minus Coleman maps}
This is an erratum to \cite[Section~5]{lei09}. We review the definition of plus and minus Coleman maps in {\it op. cit.} and describe their images explicitly. Our strategy here is based on ideas in \cite{leiloefflerzerbes11}.

Let $f$ be a modular form as in Section~\ref{sec:modularforms} with $a_p=0$ but not necessarily CM. Let $H^1_{\Iw}(\Qp,V_f)=\Qp\otimes\varprojlim H^1(\QQ_{p,n},T_f)$ (note that $V_f$ in this paper is defined to be the Tate dual of the $V_f$ in \cite{lei09}). The plus and minus Coleman maps are $\Lambda_E(G_\infty)$-homomorphisms
\[
\Col^\pm:H^1_{\Iw}(\Qp,V_f)\rightarrow\Lambda_E(G_\infty),
\]
which are defined by the relation 
\begin{equation}
\label{eq:plusminuscoleman}\Tw_1(\log_{k-1}^\pm)\Col^\pm=\LL_{\eta^\pm},
\end{equation}
 where
\[
\LL_{\eta^\pm}:H^1_{\Iw}(\Qp,V_f)\rightarrow\calH_{(k-1)/2,E}(G_\infty)
\]
are the Perrin-Riou pairing associated to $\eta^\pm\in\widetilde{\DD}_{\cris}(V_f^*(-1))$. Here, $\eta^-\in\Fil^1\widetilde{\DD}_{\cris}(V_f^*(-1))$ and $\eta^+=\vp(\eta^-)$.

Let $\mathbf{z}\in H^1_{\Iw}(\Qp,V_f)$ and write $z_{-j,n}$ for its natural image in $H^1(\QQ_{p,n},V_f(-j))$. For an integer $j\in[0,k-2]$ and a Dirichlet character $\theta$ of conductor $p^n>1$, we have
\begin{eqnarray}
\chi^j(\LL_{\eta^\pm}(\mathbf{z}))&=&j!\left[\left(1-\frac{\vp^{-1}}{p}\right)(1-\vp)^{-1}(\eta^\pm_{j+1}),\exp^*(z_{-j,0})\right]_0,\label{eq:char1}\\
\theta\chi^j(\LL_{\eta^\pm}(\mathbf{z}))&=&\frac{j!}{\tau(\theta^{-1})}\sum_{\sigma\in G_n}\theta^{-1}(\sigma)\left[\vp^{-n}(\eta_{j+1}^\pm),\exp^*(z_{-j,n}^\sigma)\right]_n.\label{eq:char2}
\end{eqnarray}
Here, $\eta_{j+1}^\pm$ denotes the natural image of $\eta^\pm$ in $\widetilde{\DD}_{\cris}(V_f^*(j))$ and $[,]_n$ is the natural pairing on
\[
\QQ_{p,n}\otimes\widetilde{\DD}_{\cris}(V_f^*(j))\times\QQ_{p,n}\otimes\widetilde{\DD}_{\cris}(V_f(-j))\rightarrow\QQ_{p,n}\otimes E.
\]

\begin{proposition}\label{prop:det}
The $\Lambda_E(G_\infty)$-homomorphism
\begin{eqnarray*}
\underline{\Col}:H^1_{\Iw}(\Qp,V_f)&\rightarrow&\Lambda_E(G_\infty)^{\oplus2}\\
\mathbf{z}&\mapsto&(\Col^+(\mathbf{z}),\Col^-(\mathbf{z}))
\end{eqnarray*}
has determinant $\prod_{j=0}^{k-2}(\chi^{-j}(\gamma_0)\gamma_0-1)$.
\end{proposition}
\begin{proof}
  The $\delta(V)$-conjecture of Perrin-Riou (\cite[Conjecture 3.4.7]{perrinriou94}), which is a consequence of \cite[Th\'eor\`eme IX.4.5]{colmez98} says that the determinant of the $\Lambda_E(G_\infty)$-homomorphism
\begin{eqnarray*}
H^1_{\Iw}(\Qp,V_f)&\rightarrow&\calH_{(k-1)/2,E}(G_\infty)^{\oplus 2}\\
\mathbf{z}&\mapsto&(\LL_{\eta^+}(\mathbf{z}),\LL_{\eta^-}(\mathbf{z}))
\end{eqnarray*}
  is given by $\prod_{j=0}^{k-2}\log_p(\chi^{-j}(\gamma_0)\gamma_0)$. But we have
\[
\prod_{j=0}^{k-2}\log_p(\chi^{-j}(\gamma_0)\gamma_0)=\log_{k-1}^+\times\log_{k-1}^-\times\prod_{j=0}^{k-2}(\chi^{-j}(\gamma_0)\gamma_0-1)
\]
by definition.  Hence the result follows from \eqref{eq:plusminuscoleman}.
\end{proof}

\begin{theorem}
The image of $\underline{\Col}$ is given by
\[
\begin{split}
S:=\{(F,G)\in\Lambda_E(G)^{\oplus2}:(\epsilon(p)^{-1}p^{1+j-k}+p^{-j-1})\chi^j(F)=(1-p^{-1})\chi^j(G),\theta\chi^j(F)=0&\\
\text{ for all integers $0\le j\le k-2$ and Dirichlet characters $\theta$ of conductor $p$}\}.&
\end{split}
\]
\end{theorem}
\begin{proof}
Since the set $S$ has determinant $\prod_{j=0}^{k-2}(\chi^{-j}(\gamma_0)\gamma_0-1)$, in view of Proposition~\ref{prop:det}, we only need to show that $\underline{\Col}(\mathbf{z})\in S$ for all $\mathbf{z}\in H^1_{\Iw}(\Qp,V_f)$.

Fix an integer $j\in[0,k-2]$. If $\theta$ is a Dirichlet character of conductor $p$, \eqref{eq:char2} says that
\[
\theta\chi^j(\LL_{\eta^+}(\mathbf{z}))=\frac{j!}{\tau(\theta^{-1})}\sum_{\sigma\in G_n}\theta^{-1}(\sigma)\left[\vp^{-1}(\eta_{j+1}^+),\exp^*(z_{-j,n}^\sigma)\right]_n.
\]
But 
\[
\vp^{-1}(\eta_{j+1}^+)=p^{j+1}\eta^-_{j+1}\in\Fil^0\widetilde{\DD}_{\cris}(V_f^*(j))
\]
and
\begin{equation}\label{eq:vanishpairing}
\left[\eta_{j+1}^-,\exp^*(z_{-j,n}^\sigma)\right]_n=0.
\end{equation}
Hence $\theta\chi^j(\LL_{\eta^+}(\mathbf{z}))=0$.

On $\widetilde{\DD}_{\cris}(V_f^*(j))$, we have $\vp^2+\epsilon(p)p^{k-2j-3}=0$, so \cite[Lemma~5.6]{leiloefflerzerbes11} implies that
\[
\left(1-\frac{\vp^{-1}}{p}\right)(1-\vp)^{-1}=\frac{(1+\epsilon(p)p^{k-2j-2})\vp+\epsilon(p)p^{k-2j-3}(p-1)}{\epsilon(p)p^{k-2j-2}(1+\epsilon(p)p^{k-2j-3})}.
\] 
We write $\delta_j=\epsilon(p)p^{k-2j-2}(1+\epsilon(p)p^{k-2j-3})$, then \eqref{eq:char1} and \eqref{eq:vanishpairing} imply that
\begin{eqnarray*}
\chi^j(\LL_{\eta^+}(\mathbf{z}))&=&\frac{j!}{\delta_j}\times\epsilon(p)p^{k-2j-3}(p-1)\times\left[\eta^+_{j+1},\exp^*(z_{-j,0})\right]_0,\\
\chi^j(\LL_{\eta^-}(\mathbf{z}))&=&\frac{j!}{\delta_j}\times\frac{1+\epsilon(p)p^{k-2j-2}}{p^{j+1}}\times\left[\eta^+_{j+1},\exp^*(z_{-j,0})\right]_0.
\end{eqnarray*}
Therefore, we indeed have
\[
(\epsilon(p)^{-1}p^{1+j-k}+p^{-j-1})\chi^j(\LL_{\eta^+}(\mathbf{z}))=(1-p^{-1})\chi^j(\LL_{\eta^-}(\mathbf{z}))
\]
as required.
\end{proof}

Note that $\epsilon(p)^{-1}p^{1+j-k}+p^{-j-1}\ne0$ unless $\epsilon(p)=-1$ and $j=k/2-1$. We therefore deduce that
\[
\im(\Col^+)=\{F\in\Lambda_E(G_\infty):\theta\chi^j(F)=0\text{ for all $j\in[0,k-2]$ and Dirichlet characters $\theta$ of conductor $p$}\}
\]
and
\[
\im(\Col^-)=\begin{cases}
\{F\in\Lambda_E(G_\infty):\chi^{k/2-1}(F)=0\}&\text{if $\epsilon(p)=-1$ and $k$ is even,}\\
\Lambda_E(G_\infty)&\text{otherwise.}
\end{cases}
\]

If $a\in\ZZ/(p-1)\ZZ$, let 
\[
\fl_{f,a}^+=\prod_{\substack{0\le j\le k-2\\ j\not\equiv a\mod {p-1}}}(\chi^{-j}(\gamma_0)\gamma_0-1)
\]
and
\[
\fl_{f,a}^-=\begin{cases}
\chi^{-k/2+1}(\gamma_0)\gamma_0-1&\text{if $\epsilon(p)=-1$, $k$ is even and $a\equiv k/2-1\mod{p-1}$,}\\
1&\text{otherwise.}
\end{cases}
\]
Then,
\[
\im(\Col^\pm)^{\omega^a}=\Lambda_{E}(\Gamma)\fl_{f,a}^\pm.
\]
\bibliographystyle{amsalpha}
\bibliography{references}
\end{document}